\documentclass[oneside, 11pt]{amsart} 
\usepackage{amsmath,amsthm,amssymb,epic}
\usepackage{eqlist,eqparbox}
\usepackage{amsfonts}
\usepackage{latexsym}
\usepackage[2emode]{psfrag}
\usepackage{amsthm}
\usepackage{amsmath}
\usepackage[all]{xy}

\begin{document}

\newcommand{\N}{\mbox {$\mathbb N $}}
\newcommand{\Z}{\mbox {$\mathbb Z $}}
\newcommand{\Q}{\mbox {$\mathbb Q $}}
\newcommand{\R}{\mbox {$\mathbb R $}}
\newcommand{\lo }{\longrightarrow }
\newcommand{\ul}{\underleftarrow }
\newcommand{\rl}{\underrightarrow }
\newcommand{\rs }{\rightsquigarrow }
\newcommand{\ra }{\rightarrow }
\newcommand{\dd }{\rightsquigarrow }
\newcommand{\ol }{\overline }
\newcommand{\la }{\langle }
\newcommand{\tr }{\triangle }
\newcommand{\xr }{\xrightarrow }
\newcommand{\de }{\delta }
\newcommand{\pa }{\partial }
\newcommand{\LR }{\Longleftrightarrow }
\newcommand{\Ri }{\Rightarrow }
\newcommand{\va }{\varphi }
\newcommand{\Den}{{\rm Den}\,}
\newcommand{\Ker}{{\rm Ker}\,}
\newcommand{\Reg}{{\rm Reg}\,}
\newcommand{\Fix}{{\rm Fix}\,}
\newcommand{\Img}{{\rm Im}\,}
\newcommand{\At}{{\rm At}\,}
\newcommand{\K}{{\rm K}\,}
\newcommand{\G}{{\rm G}\,}
\newcommand{\Idp}{{\rm Idp}\,}
\newcommand{\Id}{{\rm Id}\,}

\newtheorem{prop}{Proposition}[section] 
\newtheorem{lemma}[prop]{Lemma}
\newtheorem{cor}[prop]{Corollary}
\newtheorem{theo}[prop]{Theorem}
                       
\theoremstyle{definition}
\newtheorem{Def}[prop]{Definition}
\newtheorem{ex}[prop]{Example}
\newtheorem{exs}[prop]{Examples}
\newtheorem{Not}[prop]{Notation}
\newtheorem{Ax}[prop]{Axiom}
\newtheorem{rems}[prop]{Remarks}
\newtheorem{rem}[prop]{Remark}
\newtheorem{op}[prop]{Open problem}
\newtheorem{conj}[prop]{Conjecture}

\def\leftmark{L.C. Ciungu}
\title{Pseudo-BCI algebras with derivations}
\author{Lavinia Corina Ciungu}

\begin{abstract}

In this paper we define two types of implicative derivations on pseudo-BCI algebras, we investigate 
their properties and we give a characterization of regular implicative derivations of type II. 
We also define the notion of a $d$-invariant deductive system of a pseudo-BCI algebra $A$ proving that 
$d$ is a regular derivation of type II if and only if every deductive system on $A$ is $d$-invariant. 
It is proved that a pseudo-BCI algebra is $p$-semisimple if and only if the only regular derivation of type II 
is the identity map. Another main result consists of proving that the set of all implicative derivations 
of a $p$-semisimple pseudo-BCI algebra forms a commutative monoid with respect to function composition. 
Two types of symmetric derivations on pseudo-BCI algebras are also introduced and it is proved that in the case of 
$p$-semisimple pseudo-BCI algebras the sets of type II implicative derivations and type II symmetric derivations 
are equal. \\


\end{abstract}

\maketitle

\section{Introduction}

The notion of derivations from the analytic theory was introduced in 1957 by Posner (\cite{Pos1}) to a prime ring
$(R, +, \cdot)$ as a map $d:R\longrightarrow R$ satisfying the conditions $d(x+y)=d(x)+d(y)$ and 
$d(x\cdot y)=d(x)\cdot y+x\cdot d(y)$, for all $x, y\in R$. Since the derivation proved to be helpful for studying 
the properties of algebraic systems, this notion has been defined and studied by many authors for the cases of 
lattices (\cite{Sza1}, \cite{Fer1}, \cite{Xin1}, \cite{Xin2}) and algebras of fuzzy logic: 
MV-algebras (\cite{Als1}), \cite{Yaz1}, \cite{Ghor1}), BCI-algebras (\cite{Jun1}, \cite{Abuj1}, \cite{Abuj2}), commutative residuated lattice (\cite{He1}), BCC-algebras (\cite{Prab1}, \cite{Als2}), BE-algebras (\cite{Kim1})), basic algebras (\cite{Krna1}) and pseudo-MV algebras (\cite{Rac7}). 

The aim of this paper is to introduce the concept of derivations on pseudo-BCI algebras and to investigate their 
properties. We define the type I and type II implicative derivations on pseudo-BCI algebras, we introduce the notion 
of a regular implicative derivation and we give a characterization of regular implicative derivations of type II. 
The notion of isotone implicative derivations is also defined, and it is proved that any regular implicative 
derivation of type II is isotone. 
For an implicative derivation $d$ on a pseudo-BCK algebra $A$ we define the notion of a $d$-invariant deductive 
system of $A$ proving that $d$ is a regular implicative derivation of type II if and only if every deductive 
system on $A$ is $d$-invariant. 
We investigate the particular case of implicative derivations on the $p$-semisimple pseudo-BCI algebras and 
we prove that a pseudo-BCI algebra is $p$-semisimple if and only if the only regular derivation of type II is the identity map. Another main result consists of proving that the set of all implicative derivations 
of a $p$-semisimple pseudo-BCI algebra forms a commutative monoid with respect to function composition.
Two types of symmetric derivations on pseudo-BCI algebras are also introduced and it is proved that in the case of 
$p$-semisimple pseudo-BCI algebras the sets of type II implicative derivations and type II symmetric derivations 
are equal.

\bigskip

\section{Preliminaries}

Pseudo-BCK algebras were introduced by G. Georgescu and A. Iorgulescu in \cite{Geo15} as algebras 
with "two differences", a left- and right-difference, and with a constant element $0$ as the least element. Nowadays pseudo-BCK algebras are used in a dual form, with two implications, $\ra$ and $\rs$ and with one constant element $1$, that is the greatest element. Thus such pseudo-BCK algebras are in the "negative cone" and are also called "left-ones". Pseudo-BCK algebras were intensively studied in \cite{Ior14}, \cite{Ior15}, \cite{Ior1}, \cite{Kuhr6}, \cite{Ciu4}. 
Pseudo-BCI algebras were defined by \cite{Dud1} as generalizations of pseudo-BCK algebra and BCI-algebras, and they 
form an important tool for an algebraic axiomatization of implicational fragment of non-classical logic (\cite{Dym3}).  
In this section we recall some basic notions and results regarding pseudo-BCI algebras from \cite{Dud1}, \cite{Dym1}-\cite{Dym7}, \cite{Chaj1}-\cite{Chaj2}, \cite{Eman1}. 

\begin{Def} \label{psBE-40-10} $\rm($\cite{Chaj2}$\rm)$ A \emph{pseudo-BCI algebra} is a structure 
${\mathcal A}=(A,\ra,\rs,1)$ of type $(2,2,0)$ satisfying the following axioms, for all $x, y, z \in A:$ \\
$(psBCI_1)$ $(x \ra y) \rs [(y \ra z) \rs (x \ra z)]=1;$ \\
$(psBCI_2)$ $(x \rs y) \ra [(y \rs z) \ra (x \rs z)]=1;$ \\
$(psBCI_3)$ $1 \ra x = x;$ \\
$(psBCI_4)$ $1 \rs x = x;$ \\
$(psBCI_5)$ $(x\ra y =1$ and $y \ra x=1)$ implies $x=y$. 
\end{Def}

It is proved in \cite[Lemma 2.1]{Chaj2} that $x\ra y=1$ iff $x\rs y=1$, so that axiom $(psBCI_5)$ is equivalent to the following axiom: \\
$(psBCI^{\prime}_5)$ $(x\rs y =1$ and $y \rs x=1)$ implies $x=y$. \\
Every pseudo-BCI algebra satisfying $x\ra y=x\rs y$ for all $x, y\in A$ is a BCI-algebra, and every pseudo-BCI algebra 
satisfying $x\le 1$ for all $x\in A$ is a pseudo-BCK algebra. 
A pseudo-BCI algebra is said to be \emph{proper} if it is not a BCI-algebra and it is not a pseudo-BCK algebra.  
In a pseudo-BCI algebra $(A,\ra,\rs,1)$, one can define a binary relation $``\leq"$ by \\
$\hspace*{2cm}$ $x\leq y$ iff $x\ra y=1$ iff $x\rs y=1$, for all $x, y\in A$. \\
We will refer to $(A,\ra,\rs,1)$ by its universe $A$. 

\begin{lemma} \label{psBE-40-20} $\rm($\cite{Eman1}$\rm)$ Let $(A,\ra,\rs,1)$ be a pseudo-BCI algebra. Then the following hold for all $x, y, z\in A$: \\
$(1)$ $x\ra x=x\rs x=1;$ \\
$(2)$ $x\le (x\ra y)\rs y$ and $x\le (x\rs y)\ra y;$ \\
$(3)$ $x\ra y=1$ iff $x\rs y=1;$ \\
$(4)$ $x\le y\ra z$ iff $y\le x\rs z;$ \\
$(5)$ $x \leq y$ implies $y \ra z \leq x \ra z$ and $y \rs z \leq x \rs z;$ \\
$(6)$ $x \leq y$ implies $z \ra x \leq z \ra y$ and $z\rs x \leq z \rs y;$ \\
$(7)$ $x\ra y\le (z\ra x)\ra (z\ra y)$ and $x\rs y\le (z\rs x)\rs (z\rs y);$ \\
$(8)$ $x\ra (y\rs z)=y\rs (x\ra z);$ \\
$(9)$ if $1\le x$, then $x=1;$ \\
$(10)$ if $x\le y$ and $y\le z$, then $x\le z;$ \\
$(11)$ $x\ra 1=x\rs 1;$ \\
$(12)$ $(x\ra y)\ra 1=(x\ra 1)\rs (y\ra 1)$ and $(x\rs y)\rs 1=(x\rs 1)\ra (y\rs 1)$.
\end{lemma}

For any pseudo-BCI algebra $A$ the set $\K(A)=\{x\in A\mid x\le 1\}$ is a subalgebra of $A$ called the 
\emph{pseudo-BCK part} of $A$, since $(\K(A),\ra,\rs,1)$ is a pseudo-BCK algebra.  
An element $a$ of a pseudo-BCI algebra $A$ is called an \emph{atom} if $a\le x$ implies $x=a$, for 
all $x\in A$. Denote by $\At(A)$ the set of all atoms of $A$ and it is proved in \cite{Dym8} that 
$\At(A)=\{x\in A\mid x=(x\ra 1)\ra 1\}$. 
For any $a\in A$ we denote $V(a)=\{x\in X\mid x\le a\}$ and it is obvious that $V(a)\neq \emptyset$, since 
$a\le a$ gives $a\in V(a)$. If $a\in \At(A)$, then the set $V(a)$ is called a \emph{branch} of $A$ determined by 
the element $a$. According to \cite{Dym2} we have: $(1)$ branches determined by different elements are disjoints; 
$(2)$ a pseudo-BCI algebra is a set-theoretic union of branches; $(3)$ comparable elements are in the same branch; 
$(4)$ the elements $x$ and $y$ belong to the same branch if and only if $x\ra y\in V(1)$ or equivalently 
$x\rs y\in V(1)$. 
A pseudo-BCI algebra $A$ is \emph{$p$-semisimple} if $x\le 1$ implies $x=1$, for all $x\in A$, that is $\K(A)=\{1\}$.   

\begin{prop} \label{psBE-40-30} $\rm($\cite{Dym2},\cite{Dym5}$\rm)$ 
In any pseudo-BCI algebra $A$ the following are equivalent, for all $x, y\in A$: \\
$(a)$ $A$ is $p$-semisimple; \\
$(b)$ if $x\le y$, then $x=y;$ \\
$(c)$ $(x\ra y)\rs y=(x\rs y)\ra y=x;$ \\
$(d)$ $(x\ra 1)\rs 1=(x\rs 1)\ra 1=x;$ \\
$(e)$ $(x\ra 1)\rs y=(y\rs 1)\ra x;$ \\ 
$(f)$ $x\ra a=y\ra a$ implies $x=y;$ \\
$(g)$ $x\rs a=y\rs a$ implies $x=y;$ \\
$(h)$ $A=\At(A);$ \\
$(i)$ $(A,\cdot,^{-1},1)$ is a group, where $x\cdot y=(x\ra 1)\rs y=(y\rs 1)\ra x$, $x^{-1}=x\ra 1=x\rs 1$, 
$x\ra y=y\cdot x^{-1}$ and $x\rs y=x^{-1}\cdot y$, for all $x, y\in A$. 
\end{prop}

\begin{prop} \label{psBE-40-40} $\rm($\cite{Dym8}$\rm)$ 
In any pseudo-BCI algebra $A$ the following are equivalent, for all $a, x, y\in A$: \\
$(a)$ $a\in \At(A);$ \\
$(b)$ $(a\ra x)\rs x=(a\rs x)\ra x=a;$ \\
$(c)$ $x\ra a=(a\ra x)\rs 1;$ \\
$(d)$ $x\rs a=(a\rs x)\ra 1;$ \\
$(e)$ $x\ra a=(a\ra y)\rs (x\ra y);$ \\ 
$(f)$ $x\rs a=(a\rs y)\ra (x\rs y);$ \\
$(g)$ $x\ra a=((x\ra a)\rs y)\ra y;$ \\
$(h)$ $x\rs a=((x\rs a)\ra y)\rs y;$ \\
$(i)$ $x\ra a=(a\ra 1)\rs (x\ra 1);$ \\
$(j)$ $x\rs a=(a\rs 1)\ra (x\rs 1);$ \\ 
$(k)$ $(a\ra 1)\rs 1=(a\rs 1)\ra 1=a$.
\end{prop}

\begin{cor} \label{psBE-40-50} If $a \in \At(A)$, then $x\ra a, x\rs a\in \At(A)$, for all $x\in A$.  
\end{cor}
\begin{proof}
It follows taking $y:=1$ in Proposition \ref{psBE-40-40}$(g)$,$(h)$ and applying $(k)$. 
\end{proof}

Let $(A,\ra, \rs,1)$ be a pseudo-BCI algebra. Denote: $x\Cup_1 y=(x\ra y)\rs y$ and $x\Cup_2 y=(x\rs y)\ra y$, 
for all $x, y\in A$. 

\begin{lemma} \label{psBE-110} In any pseudo-BCI algebra $A$ the following hold for all $x, x_1, x_2, y\in A:$ \\ 
$(1)$ $1\Cup_1 x=1\Cup_2 x=1;$ \\
$(2)$ $x\le y$ iff $x\Cup_1 y=y$ iff $x\Cup_2 y=y;$ \\
$(3)$ $x\Cup_1 x=x\Cup_2 x=x;$ \\
$(4)$ if $A$ is $p$-semisimple, then $x\Cup_1 1=x\Cup_2 1=x;$ \\
$(5)$ $x\le x\Cup_1 y$ and $x\le x\Cup_2 y;$ \\
$(6)$ $x_1\le x_2$ implies $x_1\Cup_1 y\le x_2\Cup_1 y$ and $x_1\Cup_2 y\le x_2\Cup_2 y;$ \\
$(7)$ $x\Cup_1 y\ra y=x\ra y$ and $x\Cup_2 y\rs y=x\rs y$.
\end{lemma}
\begin{proof} The proof is straightforward.
\end{proof}

A pseudo-BCK algebra $A$ is said to be \emph{commutative} if it satisfies the following identities 
for all $x, y\in A$: \\ 
$(comm_1)$ $x\Cup_1 y=y\Cup_1 x$, \\
$(comm_2)$ $x\Cup_2 y=y\Cup_2 x$. \\
Note that if a pseudo-BCI algebra $A$ satisfies $(comm_1)$ and $(comm_2)$ for all $x, y\in A$, then $A$ is a 
pseudo-BCK algebra. Indeed, $x\le (x\ra 1)\rs 1=(1\ra x)\rs x=1$, for all $x\in A$. 
A pseudo-BCI algebra satisfying $(comm_1)$ and $(comm_2)$ for all $x$ and $y$ belonging to the same branch is 
called \emph{branchwise commutative}, and it is called \emph{commutative} if it satisfies the quasi-identities \\
$(comm_3)$ $x\Cup_1 y=x\Cup_2 y=x$, whenever $y\le x$. \\
It was proved in \cite{Dym1} that a pseudo-BCI algebra is commutative if and only if it is branchwise commutative. 
Moreover, each branch of a commutative pseudo-BCI algebra is a semilattice with the join $\vee$ defined by 
$x\vee y=x\Cup_1 y=x\Cup_2 y$. 
It is known that any $p$-semisimple pseudo-BCI algebra is commutative (\cite{Dym1}). \\
A subset $D$ of a pseudo-BCI algebra $A$ is called a \emph{deductive system} of $A$ if it satisfies 
the following axioms: 
$(ds_1)$ $1\in D,$ 
$(ds_2)$ $x\in D$ and $x\ra y\in D$ imply $y\in D$. \\
A subset $D$ of $A$ is a deductive system if and only if it satisfies $(ds_1)$ and the axiom: \\
$(ds^{\prime}_2)$ $x\in D$ and $x\rs y\in D$ imply $y\in D$. \\
Denote by ${\mathcal DS}(A)$ the set of all deductive systems of $A$. 
A deductive system $D$ of $A$ is \emph{proper} if $D\ne A$. 
A deductive system $D$ of a pseudo-BCK algebra $A$ is said to be \emph{compatible} if it satisfies the condition:
$(ds_3)$ for all $x, y \in A$, $x \ra y \in D$ iff $x \rs y \in D$. \\
Denote by ${\mathcal DS_c}(A)$ the set of all compatible deductive systems of $A$. 
A deductive system is \emph{closed} if it is subalgebra of $A$. 
Denote by $\mathcal{CON}(A)$ the set of all congruences of $A$. 
We say that $\theta\in \mathcal{CON}(A)$  is a \emph{relative congruence} of $A$ if the quotient algebra 
$(A/\theta,\ra,\rs,[1]_{\theta})$ is a pseudo-BCI algebra. 
It was proved in \cite{Dym4} that the relative congruences of $A$ correspond one-to-one to closed compatible deductive 
systems of $A$. 
For details regarding the deductive systems and congruence relations on a pseudo-BCI algebra we refer the reader to \cite{Dym4}.

\bigskip

\section{Implicative derivation operators on pseudo-BCI algebras}

In this section we define the type I and type II implicative derivations on pseudo-BCI algebras, and we 
investigate their properties. The notions of regular, isotone and idempotent implicative derivations are introduced, 
and a characterization of regular implicative derivations of type II is given. 
It is also proved that any regular implicative derivation of type II is isotone. 
If $d$ is a regular implicative derivation of type II on a pseudo-BCI algebra $A$ such that its kernel coincides 
with the pseudo-BCK part of $A$, then it is proved that $d$ is idempotent. 
Finally, given an implicative derivation $d$ on $A$, we define the notion of a $d$-invariant deductive system of $A$, 
and we prove that $d$ is a regular implicative derivation of type II if and only if every deductive system of $A$ 
is $d$-invariant. 

\begin{Def} \label{dop-10} Let $(A,\ra,\rs,1)$ be a pseudo-BCI algebra. A mapping $d:A\longrightarrow A$ is called an  
\emph{implicative derivation operator of type I} or a \emph{type I implicative derivation operator} or a 
\emph{type I implicative derivation} on $A$ if it satisfies the following conditions for all $x, y\in A:$ \\
$(idop_1)$ $d(x\ra y)=(x\ra d(y))\Cup_2 (d(x)\ra y)$. \\
$(idop_2)$ $d(x\rs y)=(x\rs d(y))\Cup_1 (d(x)\rs y)$. 
\end{Def}

\begin{Def} \label{dop-20} Let $(A,\ra,\rs,1)$ be a pseudo-BCI algebra. A mapping $d:A\longrightarrow A$ is called an 
\emph{implicative derivation operator of type II} or a \emph{type II implicative derivation operator} or a 
\emph{type II implicative derivation} on $A$ if it satisfies the following conditions for all $x, y\in A:$ \\
$(idop_3)$ $d(x\ra y)=(d(x)\ra y)\Cup_2 (x\ra d(y))$, \\
$(idop_4)$ $d(x\rs y)=(d(x)\rs y)\Cup_1 (x\rs d(y))$. 
\end{Def}

Let $A$ be a pseudo-BCI algebra. Denote: \\
$\hspace*{2cm}$ $\mathcal{IDOP}^{(I)}(A)$ the set of all implicative derivation operators of type I on $A$, \\ 
$\hspace*{2cm}$ $\mathcal{IDOP}^{(II)}(A)$ the set of all implicative derivation operators of type II on $A$, \\ 
$\hspace*{2cm}$ $\mathcal{IDOP}(A)=\mathcal{IDOP}^{(I)}(A)\cap \mathcal{IDOP}^{(II)}(A)$. \\ 
A map $d\in \mathcal{IDOP}(A)$ is called an \emph{implicative derivation} on $A$. \\
In what follows we will denote $dx$ instead of $d(x)$. 

\begin{ex} \label{dop-30} Let $A$ be a pseudo-BCI algebra. 
If $\Id_A:A\longrightarrow A$, defined by $\Id_A(x)=x$ for all $x\in A$, then $\Id_A\in \mathcal{IDOP}(A)$. 
\end{ex}

\begin{ex} \label{dop-30-10} 
Consider the structure $(A,\ra,\rs,1)$, where the operations $\ra$ and $\rs$ on $A=\{a,b,c,d,1\}$ 
are defined as follows:
\[
\hspace{10mm}
\begin{array}{c|ccccc}
\rightarrow & a & b & c & d & 1 \\ \hline
a & 1 & 1 & 1 & d & 1 \\
b & b & 1 & 1 & d & 1 \\
c & b & b & 1 & d & 1 \\
d & d & d & d & 1 & d \\
1 & a & b & c & d & 1
\end{array}
\hspace{10mm}
\begin{array}{c|ccccc}
\rightsquigarrow & a & b & c & d & 1 \\ \hline
a & 1 & 1 & 1 & d & 1 \\
b & c & 1 & 1 & d & 1 \\
c & a & b & 1 & d & 1 \\
d & d & d & d & 1 & d \\
1 & a & b & c & d & 1
\end{array}
.
\]
Then $(A,\ra,\rs,1)$ is a pseudo-BCI algebra (\cite{Dym6}). 
Consider the maps $d_1,d_2,d_3:A\longrightarrow A$ given in the table below:
\[
\begin{array}{c|cccccc}
 x     & a & b & c & d & 1  \\ \hline
d_1(x) & a & b & c & d & 1 \\
d_2(x) & d & d & d & 1 & d \\
d_3(x) & 1 & 1 & 1 & d & 1   
\end{array}
.   
\]
One can check that $\mathcal{IDOP}^{(I)}(A)=\mathcal{IDOP}^{(II)}(A)=\{d_1,d_2,d_3\}$. 
\end{ex}

\begin{prop} \label{dop-40} Let $A$ be a pseudo-BCI algebra. Then the following hold for all $x\in A$: \\
$(1)$ if $d\in \mathcal{IDOP}^{(I)}(A)$, then $dx=dx\Cup_1 x=dx\Cup_2 x;$ \\
$(2)$ if $d\in \mathcal{IDOP}^{(II)}(A)$, then $dx=x\Cup_1 dx=x\Cup_2 dx$ iff $d1=1$.  
\end{prop}
\begin{proof} 
$(1)$ Consider $d\in \mathcal{IDOP}^{(I)}(A)$ and $x\in A$. Applying Lema \ref{psBE-40-20}$(8)$,$(7)$ we get: \\
$\hspace*{2cm}$ $dx=d(1\ra x)=(1\ra dx)\Cup_2 (d1\ra x)=dx\Cup_2 (d1\ra x)$ \\
$\hspace*{2.5cm}$ $=(dx\rs (d1\ra x))\ra (d1\ra x)$ \\
$\hspace*{2.5cm}$ $=(d1\ra (dx\rs x))\ra (d1\ra x)$ \\
$\hspace*{2.5cm}$ $\ge (dx\rs x)\ra x=dx\Cup_2 x$. \\
On the other hand $dx\le dx\Cup_2 x$, hence $dx=dx\Cup_2 x$. \\
Similarly we have: \\
$\hspace*{2cm}$ $dx=d(1\rs x)=(1\rs dx)\Cup_1 (d1\rs x)=dx\Cup_1 (d1\rs x)$ \\
$\hspace*{2.5cm}$ $=(dx\ra (d1\rs x))\rs (d1\rs x)$ \\
$\hspace*{2.5cm}$ $=(d1\rs (dx\ra x))\rs (d1\rs x)$ \\
$\hspace*{2.5cm}$ $\ge (dx\ra x)\rs x=dx\Cup_1 x$. \\
Since $dx\le dx\Cup_1 x$, it follows that $dx=dx\Cup_1 x$. \\
$(2)$ Let $d\in \mathcal{IDOP}^{(II)}(A)$ and let $x\in A$. Then: \\
$dx=d(1\ra x)=(d1\ra x)\Cup_2 (1\ra dx)=(d1\ra x)\Cup_2 dx$ and 
$dx=d(1\rs x)=(d1\rs x)\Cup_1 (1\rs dx)=(d1\rs x)\Cup_1 dx$. 
If $d1=1$, then $dx=x\Cup_1 dx=x\Cup_2 dx$. 
Conversely, if $dx=x\Cup_1 dx=x\Cup_2 dx$ for all $x\in A$, then $d1=1\Cup_1 d1=1\Cup_2 d1=1$. 
\end{proof}

For any $x\in A$, define $\varphi: A\longrightarrow A$ by $\varphi(x)=\varphi_x=x\Cup_1 1$. 
By Lemma \ref{psBE-40-20}$(11)$ we have $\varphi_x=x\Cup_1 1=x\Cup_2 1=(x\ra 1)\ra 1=(x\rs 1)\rs 1$. 

\begin{lemma} \label{dop-50} In any pseudo-BCI algevra $A$ the following hold for any $x\in A$: \\
$(1)$ $\varphi_x\in \At(A);$ \\
$(2)$ $\varphi_x\ra x, \varphi_x\rs x\in \K(A);$ \\
$(3)$ $x\in \K(A)$ implies $\varphi_x=1$. 
\end{lemma}
\begin{proof}
$(1)$ Since by Lemma \ref{psBE-110}$(7)$, $(\varphi_x\ra 1)\rs 1=(x\Cup_1 1\ra 1)\rs 1=(x\ra 1)\rs 1=\varphi_x$ and 
$(\varphi_x\rs 1)\ra 1=(x\Cup_2 1\rs 1)\ra 1=(x\rs 1)\ra 1=\varphi_x$, applying Proposition \ref{psBE-40-40}$(b)$ it 
follows that $\varphi_x\in \At(A)$. \\
$(2)$ From $x\le \varphi_x$ we have $\varphi_x\ra 1\le x\ra 1$, that is $(\varphi_x\ra 1)\rs (x\ra 1)=1$. 
Applying Lemma \ref{psBE-40-20}$(12)$ we get $(\varphi_x\ra x)\ra 1=1$, hence $\varphi_x\ra x\le 1$, that 
is $\varphi_x\ra x\in \K(A)$. Similarly, $\varphi_x\rs x\in \K(A)$. \\
$(3)$ It is obvious.  
\end{proof}

\begin{prop} \label{dop-60} Let $A$ be a pseudo-BCI algebra and let $d_{\varphi}:A\longrightarrow A$, defined by 
$dx=\varphi_x$ for any $x\in A$. Then $d_{\varphi}\in \mathcal{IDOP}^{(I)}(A)$. 
\end{prop}
\begin{proof}
Let $A$ be a pseudo-BCI algebra and let $x, y\in A$. By Lemma \ref{dop-50}, $\varphi_y\in \At(A)$ and by 
Corollary \ref{psBE-40-50}, $x\ra \varphi_y, x\rs \varphi_y\in \At(A)$. 
Applying Lemma \ref{psBE-40-20}$(12)$ and Proposition \ref{psBE-40-40}$(b)$ we get: \\
$\hspace*{1cm}$ $d(x\ra y)=\varphi_{x\ra y}=((x\ra y)\ra 1)\rs 1=((x\ra 1)\rs (y\ra 1))\rs 1$ \\
$\hspace*{2.6cm}$ $=((x\ra 1)\rs 1)\ra ((y\ra 1)\rs 1)=((x\ra 1)\rs 1)\ra \varphi_y$ \\
$\hspace*{2.6cm}$ $=((x\ra 1)\rs 1)\ra ((\varphi_y\ra 1)\rs 1)=((x\ra 1)\rs (\varphi_y\ra 1))\rs 1$ \\
$\hspace*{2.6cm}$ $=((x\ra \varphi_y)\ra 1)\rs 1=x\ra \varphi_y 
                   =((x\ra \varphi_y)\rs (\varphi_x\ra y))\ra (\varphi_x\ra y)$ \\
$\hspace*{2.6cm}$ $=(x\ra \varphi_y)\Cup_2 (\varphi_x\ra y)=(x\ra dy)\Cup_2 (dx\ra y)$. \\
Similarly we have: \\
$\hspace*{1cm}$ $d(x\rs y)=\varphi_{x\rs y}=((x\rs y)\ra 1)\rs 1=((x\rs y)\rs 1)\ra 1$ \\
$\hspace*{2.6cm}$ $=((x\rs 1)\ra (y\rs 1))\ra 1=((x\rs 1)\ra 1)\rs ((y\rs 1)\ra 1)$ \\ 
$\hspace*{2.6cm}$ $=((x\rs 1)\ra 1)\rs \varphi_y=((x\rs 1)\ra 1)\rs ((\varphi_y\rs 1)\ra 1)$ \\ 
$\hspace*{2.6cm}$ $=((x\rs 1)\ra (\varphi_y\rs 1))\ra 1=((x\rs \varphi_y)\rs 1)\ra 1=x\rs \varphi_y$ \\                 
$\hspace*{2.6cm}$ $=((x\rs \varphi_y)\ra (\varphi_x\rs y))\rs (\varphi_x\rs y)
                   =(x\rs \varphi_y)\Cup_1 (\varphi_x\rs y)$ \\
$\hspace*{2.6cm}$ $=(x\rs dy)\Cup_1 (dx\rs y)$. \\
We conclude that $d_{\varphi}\in \mathcal{IDOP}^{(I)}(A)$. 
\end{proof}

\begin{theo} \label{dop-70} Let $A$ be a commutative pseudo-BCI algebra, and let $d_{\varphi}:A\longrightarrow A$, defined by $d_{\varphi}x=\varphi_x$ for any $x\in A$. Then $d_{\varphi}\in \mathcal{IDOP}(A)$.
\end{theo}
\begin{proof}
By Proposition \ref{dop-60}, we have $d_{\varphi}\in \mathcal{IDOP}^{(I)}(A)$, and let $x, y\in A$. 
As we mentioned in the proof of Proposition \ref{dop-60}, 
$\varphi_x, \varphi_y, x\ra \varphi_y, x\rs \varphi_y\in \At(A)$. Then we have: \\ 
$\hspace*{2cm}$   $x\ra \varphi_y=((x\ra \varphi_y)\ra 1)\rs 1=((x\ra 1)\rs (\varphi_y\ra 1))\rs 1$ \\
$\hspace*{3.4cm}$ $=((x\ra 1)\rs 1)\ra ((\varphi_y\ra 1)\rs 1)=\varphi_x\ra \varphi_y$, \\ 
hence $x\ra \varphi_y\in V(\varphi_x\ra \varphi_y)$. Moreover: \\
$\hspace*{2cm}$   $\varphi_x\ra y\le ((\varphi_x\ra y)\ra 1)\rs 1=((\varphi_x\ra 1)\rs (y\ra 1))\rs 1$ \\
$\hspace*{3.4cm}$ $=((\varphi_x\ra 1)\rs 1)\ra ((y\ra 1)\rs 1)=\varphi_x \ra \varphi_y$, \\
that is $\varphi_x\ra y\in V(\varphi_x\ra \varphi_y)$. 
Since $A$ is commutative, it follows that $A$ is branch commutative, and in the similar way to the proof 
of Proposition \ref{dop-60} we get: \\
$\hspace*{1.0cm}$ $(d_{\varphi}x \ra y)\Cup_2 (x\ra d_{\varphi}y)=(\varphi_x \ra y)\Cup_2 (x\ra \varphi_y)
                    =(x\ra \varphi_y)\Cup_2 (\varphi_x \ra y)$ \\
$\hspace*{5.2cm}$ $=\varphi_{x\ra y}=d_{\varphi}(x\ra y)$. \\
Similarly we have: \\
$\hspace*{1.0cm}$ $(d_{\varphi}x \rs y)\Cup_1 (x\rs d_{\varphi}y)=(\varphi_x \rs y)\Cup_1 (x\rs \varphi_y)=
                    (x\rs \varphi_y)\Cup_1 (\varphi_x \rs y)$ \\
$\hspace*{5.2cm}$ $=\varphi_{x\rs y}=d_{\varphi}(x\rs y)$, \\
hence $d_{\varphi}\in \mathcal{IDOP}^{(II)}(A)$. We conclude that $d_{\varphi}\in \mathcal{IDOP}(A)$. 
\end{proof}

\begin{Def} \label{dop-80} Let $A$ be a pseudo-BCI algebra $A$. A type I or a type II derivation $d$ on $A$ is 
said to be: \\
$(1)$ \emph{regular} if $d1=1;$ \\
$(2)$ \emph{isotone} if $x\le y$ implies $dx\le dy;$ \\
$(3)$ \emph{idempotent} if $d^2=d$, where $d^2=d\circ d$.  
\end{Def}

\begin{ex} \label{dop-90} Consider the pseudo-BCI algebra $A$ and its derivations from Example \ref{dop-30-10}. 
Then: \\
$(1)$ $d_{\varphi}=d_3\in \mathcal{IDOP}(A)$. \\
$(2)$ $d_1$ and $d_3$ are regular implicative derivations, while $d_2$ is not regular. \\ 
$(3)$ $d_1$, $d_2$, $d_3$ are isotone. \\
$(4)$ $d_1$ and $d_3$ are idempotent.
\end{ex}

\begin{rem} \label{dop-90-10} In Theorem \ref{dop-70} the pseudo-BCI algebra $A$ need not be commutative. 
Indeed, the pseudo-BCI algebra from Example \ref{dop-30-10} is not commutative, but 
$d_{\varphi}=d_3\in \mathcal{IDOP}(A)$.
\end{rem}

Let $A$ be a pseudo-BCI algebra. Denote: \\
$\hspace*{1cm}$ $\mathcal{RIDOP}^{(I)}(A)$ the set of all regular implicative derivation 
                operators of type I on $A$, \\ 
$\hspace*{1cm}$ $\mathcal{RIDOP}^{(II)}(A)$ the set of all regular implicative derivation 
                operators of type II on $A$, \\ 
$\hspace*{1cm}$ $\mathcal{RIDOP}(A)=\mathcal{RIDOP}^{(I)}(A)\cap \mathcal{RIDOP}^{(II)}(A)$.

\begin{prop} \label{dop-100} Let $A$ be a pseudo-BCI algebra and let $d\in \mathcal{RIDOP}^{(II)}(A)$.
Then the following hold for all $x, y\in A$: \\
$(1)$ $x\le dx;$ \\
$(2)$ $dx\ra y\le dx\ra dy\le x\ra dy=d(x\ra y)$ and $dx\rs y\le dx\rs dy\le x\rs dy=d(x\rs y);$ \\
$(3)$ $d$ is isotone; \\
$(4)$ $\Ker(d)=\{x\in A \mid dx=1\}$ is a subalgebra of $A;$ \\
$(5)$ $\Ker(d) \subseteq \K(A);$ \\ 
$(6)$ $d(\K(A))\subseteq \K(A);$ \\ 
$(7)$ $\varphi_x$ and $dx$ belong to the same branch of $A;$ \\
$(8)$ $\varphi_x\ra dx, \varphi_x\rs dx \in \K(A)$.
\end{prop}
\begin{proof} 
$(1)$ By Lemma \ref{psBE-110}$(5)$ and Proposition \ref{dop-40}$(2)$, $x\le x\Cup_1 dx= dx$. \\ 
$(2)$ From $y\le dy$ we have $dx\ra y\le dx\ra dy$, and by $x\le dx$ we get $dx\ra dy\le x\ra dy$. 
It follows that: \\
$\hspace*{2cm}$ $dx\ra y\le dx\ra dy\le x\ra dy=1\ra (x\ra dy)$ \\
$\hspace*{3.3cm}$ $=((dx\ra y)\rs (x\ra dy))\ra (x\ra dy)$ \\
$\hspace*{3.3cm}$ $=(dx\ra y)\Cup_2 (x\ra dy)=d(x\ra y)$. \\
Similarly, $dx\rs y\le dx\rs dy\le x\rs dy=d(x\rs y)$. \\
$(3)$ Let $x, y\in A$ such that $x\le y$, that is $x\ra y=1$. 
Then we have $dx\ra y\le (dx\ra y)\Cup_2 (x\ra dy)=d(x\ra y)=d1=1$. It follows that $dx\le y\le dy$, hence $d$ 
is isotone. \\
$(4)$ Since $d1=1$, it follows that $1\in \Ker(d)$. Let $x, y\in \Ker(d)$, that is $dx=dy=1$. 
Then we have $1=dx\ra dy\le d(x\ra y)$, hence $d(x\ra y)=1$, that is $x\ra y\in \Ker(d)$. 
Similarly, $x\rs y\in \Ker(d)$, thus $\Ker(d)$ is a subalgebra of $A$. \\
$(5)$ For any $x\in \Ker(d)$ we have $x\le dx=1$, that is $x\in \K(A)$. Hence $\Ker(d) \subseteq \K(A)$. \\ 
$(6)$ If $x\in d(\K(A))$, there exists $x^{\prime}\in \K(A)$ such that $x=dx^{\prime}$. 
Since $x^{\prime}\le 1$ and $x^{\prime}\le dx^{\prime}=x$, applying Lemma \ref{psBE-40-20} we get 
$x=1\ra x\le x^{\prime}\ra x=1$, that is $x\in \K(A)$. It follows that $d(\K(A))\subseteq \K(A)$. \\ 
$(7)$ From $x\le dx$, it follows that $x$ and $dx$ are in a branch $V_1$ of $A$, while from $x\le \varphi_x$ we get that $x$ and $\varphi_x$ are in a branch $V_2$ of $A$. Since $x\in V_1\cap V_2$, it follows that $V_1$ and $V_2$ 
coincide, hence $x$, $\varphi_x$ and $dx$ belong to the same branch of $A$. \\
$(8)$ According to $(7)$, there exists $a\in \At(A)$ such that $\varphi_x, dx\in V(a)$. 
From $\varphi_x\le a$, $dx\le a$ we get $\varphi_x\ra dx\le \varphi_x\ra a=1$, that is $\varphi_x\ra dx\in \K(A)$. 
Similarly, $\varphi_x\rs dx\in \K(A)$. 
\end{proof}

\begin{prop} \label{dop-120-10} Let $A$ be a pseudo-BCI algebra and let $d_1, d_2\in \mathcal{RIDOP}^{(II)}(A)$ 
such that $d_2$ is idempotent and $d_1\le d_2$ (that is $d_1x\le d_2x$, for all $x\in A$). Then $d_2\circ d_1=d_2$. 
\end{prop} 
\begin{proof} Let $x\in A$. By Proposition \ref{dop-100}$(1)$,$(3)$, $d_2x\le d_2d_1x=(d_2\circ d_1)(x)$, so 
$d_2\le d_2\circ d_1$. Moreover, since $d_1x\le d_2x$ we have $d_2d_1x\le d_2d_2x=d^2_2x=d_2x$, that is 
$d_2\circ d_1\le d_2$. Hence $d_2\circ d_1=d_2$. 
\end{proof}

\begin{prop} \label{dop-100-10} Let $A$ be a pseudo-BCI algebra and let $d\in \mathcal{RIDOP}^{(II)}(A)$. 
Then $\Ker(d)=\K(A)$ if and only if $d=d_{\varphi}$.
\end{prop}
\begin{proof} Assume that $\Ker(d)=\K(A)$ and let $x\in A$. 
Since by Lemma \ref{dop-50}, $\varphi_x\ra x\in \K(A)$, we have $d(\varphi_x\ra x)=1$. 
Applying Proposition \ref{dop-100}$(2)$ we get $1=d(\varphi_x\ra x)=\varphi_x\ra dx$, hence $\varphi_x\le dx$. 
On the other hand $\varphi_x\in \At(A)$, thus $dx=\varphi_x$. 
Conversely, assume that $dx=\varphi_x$ for all $x\in A$. Since for any $x\in \K(A)$, $dx=\varphi_x=1$, we have 
$x\in \Ker(d)$, so that $\K(A)\subseteq \Ker(d)$. 
By Proposition \ref{dop-100}$(5)$, $\Ker(d)\subseteq \K(A)$, hence $\Ker(d)=\K(A)$. 
\end{proof}

\begin{cor} \label{dop-100-20} Let $A$ be a pseudo-BCI algebra and let $d\in \mathcal{RIDOP}^{(II)}(A)$ such that  
$\Ker(d)=\K(A)$. Then $d$ is idempotent.
\end{cor}
\begin{proof} According to Proposition \ref{dop-100-10}, for any $x\in A$ we have: 
$d^2x=ddx=d_{\varphi_{dx}}=(dx\ra 1)\rs 1=(\varphi_x\ra 1)\rs 1=\varphi_x=dx$, hence $d^2=d$. 
\end{proof}

\begin{cor} \label{dop-100-30} Let $A$ be a pseudo-BCI algebra and let $d_1, d_2\in \mathcal{RIDOP}^{(II)}(A)$ 
such that $d_1\le d_2$ and $\Ker(d_1)=\K(A)$. Then $d_2\circ d_1=d_2$.
\end{cor}
\begin{proof} Since $d_1\le d_2$, we have $\Ker(d_1)\subseteq \Ker(d_2)$, that is $\Ker(d_1)=\Ker(d_2)=\K(A)$. 
According to Corollary \ref{dop-100-20}, $d_1$ and $d_2$ are idempotent, and applying Proposition \ref{dop-120-10} 
it follows that $d_2\circ d_1=d_2$. 
\end{proof}

\begin{ex} \label{dop-100-40} With the notations from Example \ref{dop-30-10}, we have: \\
$(1)$ $\K(A)=\{a,b,c,1\};$ \\
$(2)$ $\Ker(d_3)=\{a, b, c, 1\}=\K(A)$ and $d_3^2=d_3;$ \\
$(3)$ $d_1(\K(A))=\K(A)$ and $d_3(\K(A))=\{1\}\subseteq \K(A)$.  
\end{ex}

\begin{prop} \label{dop-100-70} Let $A$ be a pseudo-BCI algebra and let $d\in \mathcal{IDOP}(A)$. 
If there exists $a\in A$ such that $a\le dx$ for all $x\in A$, then: \\
$(1)$ $d\in \mathcal{RIDOP}(A);$ \\
$(2)$ $A$ is a pseudo-BCK algebra. 
\end{prop}
\begin{proof}
$(1)$ From $a\le dx$, we have $a\ra dx=a\rs dx=1$ for all $x\in A$, so $a\ra d(a\ra x)=1$ and $a\ra d1=a\rs d1=1$. 
Then we get $1=a\ra d(a\ra x)=a\ra (a\ra dx)\Cup_2 (da\ra x)=a\ra (1\Cup_2 (da\ra x))=a\ra 1$ and 
similarly, $a\rs 1=1$. It follows that 
$d1=d(a\ra 1)=(da\ra 1)\Cup_1 (a\ra d1)=(da\ra 1)\Cup_1 1=((da\ra 1)\ra 1)\rs 1=1$ (from $1\in \At(A)$ we have 
$da\ra 1\in \At(A)$, hence $(da\ra 1)\ra 1)\rs 1=1$). Hence $d1=1$, so $d\in \mathcal{RIDOP}(A)$. \\
$(2)$ Since $d$ is regular, we have $x\le dx$, thus $dx\ra 1\le x\ra 1$ for all $x\in A$. Then we have: 
$x\ra 1\ge dx\ra 1=dx\ra (a\rs dx)=a\rs (dx\ra dx)=a\rs 1=1$, hence $x\ra 1=1$. 
We conclude that $A$ is a pseudo-BCK algebra. 
\end{proof}


\begin{lemma} \label{dop-130} Let $A$ be a pseudo-BCI algebra and let $d\in \mathcal{IDOP}^{(I)}(A)$. Then: \\
$(1)$ $d1\in \At(A);$ \\
$(2)$ $da=(a\ra 1)\ra d1=(a\ra 1)\rs d1=a\cdot d1$, for all $a\in \At(A);$ \\
$(3)$ $d(dx\ra x)=d(dx\rs x)=1$, for all $x\in A$. 
\end{lemma}
\begin{proof}
$(1)$ By Proposition \ref{dop-40}$(1)$, $d1=d1\Cup_1 1=d1\Cup_2 1=(d1\ra 1)\rs 1=(d1\rs 1)\ra 1$, hence 
by Proposition \ref{psBE-40-40}$(k)$, $d1\in \At(A)$. \\
$(2)$ Applying Lemma \ref{psBE-40-20}$(8)$,$(7)$,$(2)$ we have: \\
$\hspace*{2cm}$   $da=d((a\ra 1)\rs 1)=((a\ra 1)\rs d1)\Cup_1 (d(a\ra 1)\rs 1)$ \\
$\hspace*{2.5cm}$ $=(((a\ra 1)\rs d1)\ra (d(a\ra 1)\rs 1))\rs (d(a\ra 1)\rs 1)$ \\
$\hspace*{2.5cm}$ $=(d(a\ra 1)\rs (((a\ra 1)\rs d1)\ra 1))\rs (d(a\ra 1)\rs 1)$ \\
$\hspace*{2.5cm}$ $\ge (((a\ra 1)\rs d1)\ra 1)\rs 1\ge (a\ra 1)\rs d1$. \\
Since by $(1)$, $d1\in \At(A)$, according to Corollary \ref{psBE-40-50}, $(a\ra 1)\rs d1\in \At(A)$.  
It follows that $da=(a\ra 1)\rs d1$. Similarly, $da=(a\ra 1)\ra d1$. Obviously $da=a\cdot d1$. \\
$(3)$ For any $x\in A$ we have $d(dx\ra x)=(dx\ra dx)\Cup_2 (x\ra ddx)=1\Cup_2(x\ra ddx)=1$ and 
similarly, $d(dx\rs x)=1$. 
\end{proof}

\begin{lemma} \label{dop-130-10} Let $A$ be a pseudo-BCI algebra and let $d\in \mathcal{IDOP}^{(II)}(A)$. Then: \\
$(1)$ $d1\ra x\le dx$ and $d1\rs x\le dx$, for all $x\in A;$ \\
$(2)$ $da=d1\ra a=d1\rs a=d1\cdot a$, for all $a\in \At(A)$. 
\end{lemma}
\begin{proof}
$(1)$ For all $x\in A$ we have $dx=d(1\ra x)=(d1\ra x)\Cup_2 dx\ge d1\ra x$ and 
$dx=d(1\rs x)=(d1\rs x)\Cup_1 dx\ge d1\rs x$. \\
$(2)$ From $d1\ra a\le da$, since $a\in \At(A)$ we get $d1\ra a\in \At(A)$, hence $da=d1\ra a$. 
Similarly, $da=d1\rs a$. Since $d1=d1\ra 1$, we have $da=d1\rs a=(d1\ra 1)\rs a=d1\cdot a$.   
\end{proof}

\begin{prop} \label{dop-140} Let $A$ be a pseudo-BCI algebra and let 
$d\in \mathcal{IDOP}^{(I)}(A)\cup \mathcal{IDOP}^{(II)}(A)$. Then: \\
$(1)$ $d(\At(A))\subseteq \At(A);$\\
$(2)$ if $d1\in \K(A)$, then $d1=1;$ \\
$(3)$ $d(a\cdot b)=da\cdot (d1\ra 1)\cdot db$, for all $a, b\in \At(A);$ \\
$(4)$ $d_{\mid \At(A)}=Id_{\At(A)}$ iff $d1=1$. 
\end{prop}
\begin{proof}
$(1)$ Let $a\in \At(A)$ and let $d\in \mathcal{IDOP}^{(I)}(A)$, by Lemma \ref{dop-130}$(1)$, $d1\in \At(A)$, so 
$(a\ra 1)\ra d1\in \At(A)$.  Hence by Lemma \ref{dop-130}$(2)$, $da\in \At(A)$. 
If $d\in \mathcal{IDOP}^{(II)}(A)$, the assertion follows applying Lemma \ref{dop-130-10}$(2)$. \\ 
$(2)$ Since $d1\le 1$ and by $(1)$, $d1\in \At(A)$, we get $d1=1$. \\
$(3)$ Let $a, b\in \At(A)$, that is $a\cdot b=(a\ra 1)\rs b\in \At(A)$. 
If $d\in \mathcal{IDOP}^{(I)}(A)$, by Lemma \ref{dop-130}$(2)$ we have: \\
$\hspace*{2cm}$ $a\cdot b=(a\cdot b)\cdot d1=(a\cdot d1)\cdot (d1)^{-1}\cdot (b\cdot d1)=da\cdot (d1\ra 1)\cdot db$. \\
For the case of $d\in \mathcal{IDOP}^{(II)}(A)$, applying Lemma \ref{dop-130-10}$(2)$ we get: \\
$\hspace*{2cm}$ $a\cdot b=d1\cdot (a\cdot b)=(d1\cdot a)\cdot (d1)^{-1}\cdot (d1\cdot b)=da\cdot (d1\ra 1)\cdot db$. \\ 
$(4)$ If $d_{\mid \At(A)}=Id_{\At(A)}$, then $d1=1$. 
Conversely, if $d\in \mathcal{IDOP}^{(I)}(A)\cup \mathcal{IDOP}^{(II)}(A)$ such that $d1=1$, then for all $x\in \At(A)$ 
we have $dx=x\cdot d1=x\cdot 1=x$ and $dx=d1\cdot x=1\cdot x=x$, respectively. Hence $d_{\mid \At(A)}=Id_{\At(A)}$.  
\end{proof}

\begin{ex} \label{dop-150} Let $A$ be the pseudo-BCI algebra and its derivations from Example \ref{dop-30-10}. 
We can see that $\At(A)=\{d, 1\}$, $\mathcal{RIDOP}^{(II)}(A)=\{d_1, d_3\}$ and 
$d_{{1}_{\mid \At(A)}}=d_{{3}_{\mid \At(A)}}=Id_{\At(A)}$. 
Moreover, for each $d\in \{d_1, d_2, d_3\}$ we have $d(\At(A))=\{d, 1\}=\At(A)$. 
\end{ex}

\begin{theo} \label{dop-110} Let $A$ be a pseudo-BCI algebra and let $d:A\longrightarrow A$. 
Then $d\in \mathcal{RIDOP}^{(II)}(A)$ if and only if $d1=1$ and $d(x\ra y)=x\ra dy$, $d(x\rs y)=x\rs dy$, 
for all $x, y\in A$.
\end{theo}
\begin{proof}
Let $d\in \mathcal{RIDOP}^{(II)}(A)$, and let $x, y\in A$. 
Clearly $d1=1$, and according to Proposition \ref{dop-100}$(2)$, $d(x\ra y)=x\ra dy$ and $d(x\rs y)=x\rs dy$. 
Conversely, let $d:A\longrightarrow A$ such that $d1=1$ and $d(x\ra y)=x\ra dy$, $d(x\rs y)=x\rs dy$, 
for all $x, y\in A$. Taking $x:=d1$, $y:=1$ we get $d1=d(d1\ra 1)=d1\ra d1=1$. 
For $y:=x$, we have $x\ra dx=d(x\ra x)=d1=1$, so $x\le dx$ for all $x\in A$. 
From $x\le dx$ and $y\le dy$ we get $dx\ra y\le x\ra y\le x\ra dy$ and $dx\rs y\le x\rs y\le x\rs dy$. 
It follows that $d(x\ra y)=x\ra dy=1\ra (x\ra dy)=((dx\ra y)\rs (x\ra dy))\ra (x\ra dy)=(dx\ra y)\Cup_2 (x\ra dy)$. 
Similarly, $d(x\rs y)=(dx\rs y)\Cup_1 (x\rs dy)$, hence $d\in \mathcal{RIDOP}^{(II)}(A)$. 
\end{proof}

\begin{prop} \label{dop-110-20} Let $A$ be a pseudo-BCI algebra and let $d\in \mathcal{IDOP}^{(I)}(A)$ such that 
$\Img(d)\subseteq \At(A)$. Then $d(x\ra y)=x\ra dy$ and $d(x\rs y)=x\rs dy$, for all $x, y\in A$.
\end{prop}
\begin{proof}
Let $x, y\in A$. Since $dy\in \At(A)$, according to Corollary \ref{psBE-40-50}, $x\ra dy\in \At(A)$. 
From $x\ra dy\le (x\ra dy)\Cup_2 (dx\ra y)=d(x\ra y)$, it follows that $d(x\ra y)=x\ra dy$. 
Similarly, $d(x\rs y)=x\rs dy$. 
\end{proof}

\begin{rem} \label{dop-120} If $d\in \mathcal{RIDOP}^{(I)}(A)\cup \mathcal{RIDOP}^{(II)}(A)$ such that 
$d(x\ra y)=dx\ra y$ or $d(x\rs y)=dx\rs y$ for all $x, y\in A$, then $d=\Id_A$. 
Indeed, $dx=d(1\ra x)=d1\ra x=1\ra x=x$ for all $x\in A$, that is $d=\Id_A$. Similarly for $d(x\rs y)=dx\rs y$. 
\end{rem}

\begin{Def} \label{dop-160} Let $A$ be a pseudo-BCI algebra and let $d\in \mathcal{IDOP}(A)$. 
Then $D\in \mathcal {DS}(A)$ is said to be \emph{$d$-invariant} if $d(D)\subseteq D$. 
\end{Def}

\begin{prop} \label{dop-170} Let $A$ be a pseudo-BCI algebra and let 
$d\in \mathcal{IDOP}^{(I)}(A)\cup \mathcal{IDOP}^{(II)}(A)$. If every deductive system of $A$ is $d$-invariant, 
then $d$ is regular.  
\end{prop}
\begin{proof}
Let $d\in \mathcal{IDOP}^{(I)}(A)\cup \mathcal{IDOP}^{(II)}(A)$ and assume that every 
$D\in \mathcal {DS}(A)$ is $d$-invariant. Since $\{1\}\in \mathcal {DS}(A)$, it follows that 
$d(\{1\})\subseteq \{1\}$, hence $d(1)=1$, that is $d$ is 
regular. 
\end{proof}

\begin{prop} \label{dop-180} Let $A$ be a pseudo-BCI algebra and let $d\in \mathcal{RIDOP}^{(II)}(A)$. 
Then every deductive system of $A$ is $d$-invariant. 
\end{prop}
\begin{proof}
Assume that $d\in \mathcal{RIDOP}^{(II)}(A)$. Let $D\in \mathcal {DS}(A)$ and let $y\in d(D)$, that is there exists 
$x\in D$ such that $y=dx$. Since by Proposition \ref{dop-100}, $x\ra y=x\ra dx=1\in D$, it follows that $y\in D$. 
Hence $d(D)\subseteq D$, that is $D$ is $d$-invariant. 
\end{proof}

\begin{cor} \label{dop-190} Let $A$ be a pseudo-BCI algebra and let $d\in \mathcal{IDOP}^{(II)}(A)$. 
Then $d\in \mathcal{RIDOP}^{(II)}(A)$ if and only if every deductive system of $A$ is $d$-invariant. 
\end{cor}

\begin{ex} \label{dop-200} Let $A$ be the pseudo-BCI algebra and its derivations from Example \ref{dop-30-10}. 
We can see that $\mathcal {DS}(A)=\{\{1\},\{c,1\},\{a,b,c,1\},A\}$ and consider $d_2\in \mathcal{IDOP}^{(II)}(A)$, but $d_2\notin \mathcal{RIDOP}^{(II)}(A)$. For $D=\{c,1\}\in \mathcal {DS}(A)$, one can easily check that 
$d_2(D)=\{d\}\nsubseteq D$, hence $D$ is not $d_2$-invariant.  
On the other hand, for $d_1, d_3\subseteq \mathcal{RIDOP}^{(II)}(A)$ we have $d_1(D)\subseteq D$ 
and $d_3(D)\subseteq D$, so that $D$ is $d_1$-invariant and $d_3$-invariant. 
\end{ex} 

\begin{ex} \label{dop-210} Let $A$ be a pseudo-BCI algebra and let $d\in \mathcal{IDOP}^{(II)}(A)$.
According to \cite[Prop. 12]{Dym4}, $\K(A)$ is a closed compatible deductive system of $A$. 
By Proposition \ref{dop-100}, $d(\K(A))\subseteq \K(A)$, hence $\K(A)$ is $d$-invariant. 
\end{ex}

\bigskip

\section{Implicative derivation operators on $p$-semisimple pseudo-BCI algebras}

In this section we investigate the particular case of implicative derivations on the $p$-semisimple 
pseudo-BCI algebras, and we prove that a pseudo-BCI algebra $A$ is $p$-semisimple if and only if the only regular 
derivation of type II on $A$ is the identity map. It is also proved that a pseudo-BCI algebra $A$ is $p$-semisimple 
if and only if the kernel of any regular implicative derivation $d$ of type II on $A$ is the set $\{1\}$. 
As a corollary it is proved that, for any pseudo-BCI algebra $A$, the only regular implicative derivation of 
type II on $A/\K(A)$ is the identity map.
It is also proved that the set of all implicative derivations of a $p$-semisimple pseudo-BCI algebra forms a commutative monoid with respect to function composition.

\begin{theo} \label{pdop-10} If $A$ is a pseudo-BCI algebra, the following are equivalent: \\
$(a)$ $A$ is $p$-semisimple; \\
$(b)$ $\Ker(d)=\{1\}$ for all $d\in \mathcal{RIDOP}^{(II)}(A);$ \\
$(c)$ $\mathcal{RIDOP}^{(II)}(A)=\{\Id_A\}$.
\end{theo}
\begin{proof}
$(a)\Rightarrow (b)$ Let $A$ be a $p$-semisimple pseudo-BCI algebra, that is $\K(A)=\{1\}$, and let 
$d\in \mathcal{RIDOP}^{(II)}(A)$. According to Proposition \ref{dop-100}$(5)$, $\Ker(d) \subseteq \K(A)$, hence $\Ker(d)=\{1\}$. \\
$(b)\Rightarrow (a)$ Assume that $\Ker(d)=\{1\}$ for all $d\in \mathcal{RIDOP}^{(II)}(A)$. 
Consider $d_{\varphi}:A\longrightarrow A$, defined by $d_{\varphi}x=\varphi_x$ for any $x\in A$. 
Since any $p$-semisimple pseudo-BCI algebra is commutative, applying Theorem \ref{dop-70} it follows that 
$d_{\varphi}\in \mathcal{IDOP}^{(II)}(A)$. Taking into consideration that $d_{\varphi}1=1$, we have 
$d_{\varphi}\in \mathcal{RIDOP}^{(II)}(A)$, so that $\Ker(d_{\varphi})=\{1\}$. 
Moreover, for any $x\in \K(A)$, $d_{\varphi}x=(x\ra 1)\rs 1=1$, that is $x\in \Ker(d_{\varphi})$. 
Thus $\K(A)\subseteq \Ker(d_{\varphi})$, and applying Proposition \ref{dop-100}$(5)$ we get 
$\K(A)=\Ker(d_{\varphi})=\{1\}$, hence $A$ is $p$-semisimple. \\
$(b)\Rightarrow (c)$ If $d\in \mathcal{RIDOP}^{(II)}(A)$, then by $(b)$ we have $\Ker(d)=\{1\}$. 
Since $(b)\Rightarrow (a)$ and by Proposition \ref{psBE-40-30}, $\At(A)=A$, applying Proposition \ref{dop-140} 
we get $d=\Id_A$. \\
$(c)\Rightarrow (b)$ Let $d\in \mathcal{RIDOP}^{(II)}(A)$, so by $(c)$ we have $d=\Id_A$, that is $\Ker(d)=\{1\}$. 
\end{proof}

\begin{cor} \label{pdop-10-10}
If $A$ is a pseudo-BCI algebra, then $\mathcal{RIDOP}^{(II)}(\At(A))=\{\Id_A\}$.
\end{cor} 
\begin{proof}
According to \cite[Th. 4.13]{Dym2}, $\At(A)$ is a $p$-semisimple pseudo-BCI subalgebra of $A$, and by 
Theorem \ref{pdop-10} it follows that $\mathcal{RIDOP}^{(II)}(\At(A))=\{\Id_A\}$.
\end{proof}

\begin{cor} \label{pdop-10-20}
If $A$ is a pseudo-BCI algebra, then $\mathcal{RIDOP}^{(II)}(A/\K(A))=\{\Id_A\}$.
\end{cor}
\begin{proof}
According to \cite[Prop.12, Th. 6]{Dym4}, $\K(A)$ is a closed compatible deductive system of $A$ and $A/\K(A)$ is a $p$-semisimple pseudo-BCI algebra. Applying Theorem \ref{pdop-10}, it follows that \\ $\mathcal{RIDOP}^{(II)}(A/\K(A))=\{\Id_A\}$. 
\end{proof}

\begin{ex} \label{pdop-10-30} 
Consider the structure $(A,\ra,\rs,1)$, where the operations $\ra$ and $\rs$ on $A=\{a,b,c,d,e,1\}$ 
are defined as follows:
\[
\hspace{10mm}
\begin{array}{c|cccccc}
\rightarrow & a & b & c & d & e & 1 \\ \hline
a & 1 & d & e & b & c & a \\
b & c & 1 & a & e & d & b \\
c & e & a & 1 & c & b & d \\
d & b & e & d & 1 & a & c \\
e & d & c & b & a & 1 & e \\
1 & a & b & c & d & e & 1 
\end{array}
\hspace{10mm}
\begin{array}{c|cccccc}
\rightsquigarrow & a & b & c & d & e & 1 \\ \hline
a & 1 & c & b & e & d & a \\
b & d & 1 & e & a & c & b \\
c & b & e & 1 & c & a & d \\
d & e & a & d & 1 & b & c \\
e & c & d & a & b & 1 & e \\
1 & a & b & c & d & e & 1 \end{array}
.
\]
Then $(A,\ra,\rs,1)$ is a $p$-semisimple pseudo-BCI algebra (\cite{Dym4}), and one can check that: \\
$(1)$ $\At(A)=A$, $\K(A)=\{1\}$ and $A/\K(A)=A;$  \\
$(2)$ $\mathcal{RIDOP}^{(II)}(A)=\mathcal{RIDOP}^{(II)}(\At(A))=\mathcal{RIDOP}^{(II)}(A/\K(A))=\{\Id_A\}$.
\end{ex}

\begin{ex} \label{pdop-10-40} Let $A$ be the pseudo-BCI algebra and its derivations from Example \ref{dop-30-10}. 
One can easily check that: \\ 
$(1)$ $\At(A)=\{d, 1\};$ \\
$(2)$ $\At(A)$ is a $p$-semisimple subalgebra of $A;$ \\
$(3)$ $\mathcal{RIDOP}^{(II)}(\At(A))=Id_{\At(A)}$.
\end{ex}

\begin{rem} \label{pdop-100} The notion of a \emph{$\ra$medial pseudo-BCI algebra} was defined in \cite{Jun2} 
as a pseudo-BCI algebra $A$ satisfying the identity $(u\rs v)\ra (x\rs y)=(u\rs x)\ra (v\rs y)$, for all 
$x, y, u, v\in A$. It was proved in \cite{Dym2} that a $\ra$medial pseudo-BCI algebra is a $p$-semisimple BCI-algebra, 
so that by Theorem \ref{pdop-10}, $\mathcal{RIDOP}^{(II)}(A)=\{\Id_A\}$. 
Similarly for the case of a \emph{$\rs$medial pseudo-BCI algebra} which is a pseudo-BCI algebra $A$ satisfying 
the identity $(u\ra v)\rs (x\ra y)=(u\ra x)\rs (v\ra y)$, for all $x, y, u, v\in A$ (\cite{Lee1}).
\end{rem}

\begin{prop} \label{dop-300} Let $A$ be a $p$-semisimple pseudo-BCI algebra and let 
$d_1, d_2\in \mathcal{IDOP}^{(I)}(A)$. Then $d_1\circ d_2\in \mathcal{IDOP}^{(I)}(A)$. 
\end{prop}
\begin{proof}
According to Proposition \ref{psBE-40-30}$(c)$, $x\Cup_1 y=x\Cup_2 y=x$ for all $x, y\in A$.\\  
If $d_1, d_2\in \mathcal{IDOP}^{(I)}(A)$ we have: \\
$\hspace*{2cm}$   $(d_1\circ d_2)(x\ra y)=d_1d_2(x\ra y)=d_1((x\ra d_2y)\Cup_2 (d_2x\ra y))$ \\
$\hspace*{4.8cm}$ $=d_1(x\ra d_2y)=(x\ra d_1d_2y)\Cup_2 (d_1x\ra d_2y)$ \\
$\hspace*{4.8cm}$ $=x\ra d_1d_2y=(x\ra d_1d_2y)\Cup_2 (d_1d_2x\ra y)$ \\
$\hspace*{4.8cm}$ $=(x\ra (d_1\circ d_2)(y))\Cup_2 ((d_1\circ d_2)(x)\ra y)$. \\
Similarly, $(d_1\circ d_2)(x\rs y)=(x\rs (d_1\circ d_2)(y))\Cup_1 ((d_1\circ d_2)(x)\rs y)$, 
hence $d_1\circ d_2\in \mathcal{IDOP}^{(I)}(A)$. 
\end{proof}

\begin{prop} \label{dop-310} Let $A$ be a $p$-semisimple pseudo-BCI algebra and let 
$d_1, d_2\in \mathcal{IDOP}^{(II)}(A)$. Then $d_1\circ d_2\in \mathcal{IDOP}^{(II)}(A)$. 
\end{prop}
\begin{proof}
Let $d_1, d_2\in \mathcal{IDOP}^{(II)}(A)$. Similar to Proposition \ref{dop-300} we get: \\
$\hspace*{2cm}$   $(d_1\circ d_2)(x\ra y)=d_1d_2(x\ra y)=d_1((d_2x\ra y)\Cup_2 (x\ra d_2y))$ \\
$\hspace*{4.8cm}$ $=d_1(d_2x\ra y)=(d_1d_2x\ra y)\Cup_2 (d_2x\ra d_1y)$ \\
$\hspace*{4.8cm}$ $=d_1d_2x\ra y=(d_1d_2x\ra y)\Cup_2 (x\ra d_1d_2y)$ \\
$\hspace*{4.8cm}$ $=((d_1\circ d_2)(x)\ra y)\Cup_2 (x\ra (d_1\circ d_2)(y))$. \\
Similarly, $(d_1\circ d_2)(x\rs y)=((d_1\circ d_2)(x)\rs y)\Cup_1 (x\rs (d_1\circ d_2)(y))$, 
thus $d_1\circ d_2\in \mathcal{IDOP}^{(II)}(A)$. 
\end{proof}

\begin{prop} \label{dop-320} Let $A$ be a $p$-semisimple pseudo-BCI algebra and let 
$d_1, d_2\in \mathcal{IDOP}(A)$. Then $d_1\circ d_2=d_2\circ d_1$. 
\end{prop}
\begin{proof}
Let $d_1, d_2\in \mathcal{IDOP}(A)$. Using the definitions of type I and type II derivations we get: \\
$\hspace*{2cm}$   $(d_1\circ d_2)(x)=(d_1\circ d_2)(1\ra x)=d_1((1\ra d_2x)\Cup_2 (d_21\ra x))$ \\
$\hspace*{4cm}$   $=d_1(1\ra d_2x)=(d_11\ra d_2x)\Cup_2 (1\ra d_1d_2x)=d_11\ra d_2x$. \\
$\hspace*{2cm}$   $(d_2\circ d_1)(x)=(d_2\circ d_1)(1\ra x)=d_2((d_11\ra x)\Cup_2 (1\ra d_1x))$ \\
$\hspace*{4cm}$   $=d_2(d_11\ra x)=(d_11\ra d_2x)\Cup_2 (d_2d_11\ra x)=d_11\ra d_2x$. \\
Hence $(d_1\circ d_2)(x)=(d_2\circ d_1)(x)$ for all $x\in A$. It follows that $d_1\circ d_2=d_2\circ d_1$. 
\end{proof}

\begin{theo} \label{dop-330} If $A$ is a $p$-semisimple pseudo-BCI algebra then 
$(\mathcal{IDOP}(A), \circ, \Id_A)$ is a commutative monoid.  
\end{theo}
\begin{proof}
It follows by Propositions \ref{dop-300}, \ref{dop-310}, \ref{dop-320} taking into consideration that the 
composition of functions is always associative. 
\end{proof}

\begin{rem} \label{dop-340} In Theorem \ref{dop-330} the pseudo-BCI algebra $A$ need not be $p$-semisimple. 
Indeed, consider the pseudo-BCI algebra $A$ and its derivations from Example \ref{dop-30-10}. 
Since $a\le b$ and $a\neq b$, the condition $(b)$ from Proposition \ref{psBE-40-30} is not satisfied, so $A$ 
is not $p$-semisimple. On the other hand we have $d_1\circ d_2=d_2\circ d_1=d_2$, $d_1\circ d_3=d_3\circ d_1=d_3$, 
$d_2\circ d_3=d_3\circ d_2=d_2$, hence $(\mathcal{IDOP}(A), \circ, \Id_A)$ is a commutative monoid.    
\end{rem}

\begin{prop} \label{dop-350} Let $A$ be a $p$-semisimple pseudo-BCI algebra and let 
$d_1, d_2\in \mathcal{IDOP}(A)$. Define $(d_1\ra d_2)(x)=d_1x\ra d_2x$ and $(d_1\rs d_2)(x)=d_1x\rs d_2x$, 
for all $x\in A$. Then $d_1\ra d_2=d_2\ra d_1$ and $d_1\rs d_2=d_2\rs d_1$. 
\end{prop}
\begin{proof}
Let $d_1, d_2\in \mathcal{IDOP}(A)$. We use again the fact that in any $p$-semisimple pseudo-BCI algebra $A$ 
we have $x\Cup_1 y=x\Cup_2 y=x$ for all $x, y\in A$. Applying the definitions of type I and type II 
derivations we get: \\
$\hspace*{2cm}$   $(d_1\circ d_2)(1)=(d_1\circ d_2)(x\ra x)=d_1((x\ra d_2x)\Cup_2 (d_2x\ra x))$ \\
$\hspace*{4cm}$   $=d_1(x\ra d_2x)=(d_1x\ra d_2x)\Cup_2 (x\ra d_1d_2x)=d_1x\ra d_2x$. \\
$\hspace*{2cm}$   $(d_1\circ d_2)(x)=(d_1\circ d_2)(x\ra x)=d_1((d_2x\ra x)\Cup_2 (x\ra d_2x))$ \\
$\hspace*{4cm}$   $=d_1(d_2x\ra x)=(d_2x\ra d_1x)\Cup_2 (d_1d_2x\ra x)=d_2x\ra d_1x$. \\
Hence $(d_1\ra d_2)(x)=(d_2\ra d_1)(x)=(d_1\circ d_2)(x)$ for all $x\in A$. 
It follows that $d_1\ra d_2=d_2\ra d_1$. Similarly, $d_1\rs d_2=d_2\rs d_1$. 
\end{proof}

\bigskip

\section{Symmetric derivation operators on pseudo-BCI algebras}

In this section we give a generalization of the concept of \emph{left derivation} introduced in \cite{Abuj2} for BCI-algebras. Two types of symmetric derivations on pseudo-BCI algebras are defined and studied, and the 
relationship between implicative and symmetric derivations is investigated. 
We prove that for the case of a $p$-semisimple pseudo-BCI algebra the sets of type II implicative derivations 
and type II symmetric derivations are equal, while for a $p$-semisimple BCI-algebra this result is also valid for 
the sets of type I implicative derivations and type I symmetric derivations. Finally, we show that a type I or type II symmetric derivation $d$ on a pseudo-BCI algebra $A$ is regular if and only if every deductive system of $A$ is $d$-invariant. 

\begin{Def} \label{sdop-10} Let $(A,\ra,\rs,1)$ be a pseudo-BCI algebra. A mapping $d:A\longrightarrow A$ is called a  
\emph{symmetric derivation operator of type I} or a \emph{type I symmetric derivation operator} or a 
\emph{type I symmetric derivation} on $A$ if it satisfies the following conditions for all $x, y\in A:$ \\
$(sdop_1)$ $d(x\ra y)=(x\ra d(y))\Cup_2 (y\ra d(x))$. \\
$(sdop_2)$ $d(x\rs y)=(x\rs d(y))\Cup_1 (y\rs d(x))$. 
\end{Def}

\begin{Def} \label{sdop-20} Let $(A,\ra,\rs,1)$ be a pseudo-BCI algebra. A mapping $d:A\longrightarrow A$ is 
called a \emph{symmetric derivation operator of type II} or a \emph{type II symmetric derivation operator} 
or a \emph{type II symmetric derivation} on $A$ if it satisfies the following conditions for all $x, y\in A:$ \\
$(sdop_3)$ $d(x\ra y)=(d(x)\ra y)\Cup_2 (d(y)\ra x)$. \\
$(sdop_4)$ $d(x\rs y)=(d(x)\rs y)\Cup_1 (d(y)\rs x)$. 
\end{Def}

Let $A$ be a pseudo-BCI algebra. Denote: \\
$\hspace*{2cm}$ $\mathcal{SDOP}^{(I)}(A)$ the set of all symmetric derivation operators of type I on $A$, \\ 
$\hspace*{2cm}$ $\mathcal{SDOP}^{(II)}(A)$ the set of all symmetric derivation operators of type II on $A$, \\ 
$\hspace*{2cm}$ $\mathcal{SDOP}(A)=\mathcal{SDOP}^{(I)}(A)\cap \mathcal{SDOP}^{(II)}(A)$. \\ 
A map $d\in \mathcal{SDOP}(A)$ is called a \emph{symmetric derivation} on $A$. \\
In what follows we will denote $dx$ instead of $d(x)$. 
Similar to the case of implicative derivations we can define the regular symmetric derivations on pseudo-BCI algebras. 
Denote by $\mathcal{RSDOP}^{(I)}(A)$, $\mathcal{RSDOP}^{(II)}(A)$ and $\mathcal{RSDOP}(A)$ 
the set of all regular symmetric derivations from $\mathcal{SDOP}^{(I)}(A)$, $\mathcal{SDOP}^{(II)}(A)$ and 
$\mathcal{SDOP}(A)$, respectively. 

\begin{prop} \label{sdop-30} Let $(A,\ra,\rs,1)$ be a pseudo-BCI algebra and let $d\in \mathcal{SDOP}^{(I)}(A)$. 
The following hold for all $x, y\in A:$ \\
$(1)$ $d1=x\ra dx=x\rs dx;$ \\
$(2)$ $dx=dx\Cup_1 d1=dx\Cup_2 d1;$ \\
If $d$ is regular, then: \\
$(3)$ $x\le dx;$ \\
$(4)$ $dx\in \At(A);$ \\
$(5)$ $dx=dx\Cup_1 y=dx\Cup_2 y$. 
\end{prop}
\begin{proof}
$(1)$ We have $d1=d(x\ra x)=(x\ra dx)\Cup_2 (x\ra dx)=x\ra dx$ and similarly, $d1=x\rs dx$. \\
$(2)$ Using Lemmas \ref{psBE-40-20}$(7)$,$(8)$ and \ref{psBE-110}$(5)$ we get: \\
$\hspace*{2cm}$   $dx=d(1\rs x)=dx\Cup_1(x\rs d1)=(dx\ra (x\rs d1))\rs (x\rs d1)$ \\
$\hspace*{4.5cm}$ $=(x\rs (dx\ra d1))\rs (x\rs d1)$ \\
$\hspace*{4.5cm}$ $\ge (dx\ra d1)\rs d1=dx\Cup_1 d1\ge dx$. \\ 
Hence $dx=dx\Cup_1 d1$ and similarly, $dx=dx\Cup_2 d1$. \\
$(3)$ It follows from $(1)$, since $d1=1$. \\
$(4)$ Since $d1=1$, applying $(2)$ we get $dx=(dx\ra 1)\rs1=(dx\rs 1)\ra 1$, hence by 
Proposition \ref{psBE-40-40}$(k)$, $dx\in \At(A)$. \\ 
$(5)$ It follows by Proposition \ref{psBE-40-40}$(b)$, since $dx\in \At(A)$. 
\end{proof}

\begin{prop} \label{sdop-40} Let $(A,\ra,\rs,1)$ be a pseudo-BCI algebra and let $d\in \mathcal{SDOP}^{(II)}(A)$. 
The following hold for all $x, y\in A:$ \\
$(1)$ $d1=dx\ra x=dx\rs x;$ \\
$(2)$ $dx=\varphi_{dx\Cup_1 x}=\varphi_{dx\Cup_2 x};$ \\
$(3)$ $dx=dx\Cup_1 x=dx\Cup_2 x;$ \\
$(4)$ $dx\in \At(A);$ \\
$(5)$ if $d$ is regular, then $d=\Id_A$.  
\end{prop}
\begin{proof}
$(1)$ We have $d1=d(x\ra x)=(dx\ra x)\Cup_2 (dx\ra x)=dx\ra x$ and similarly, $dx=dx\rs x$. \\
$(2)$ Applying Lemma \ref{psBE-40-20}$(7)$,$(8)$ and $(1)$ we get: \\
$\hspace*{2cm}$   $dx=d(1\rs x)=(d1\rs x)\Cup_1 (dx\rs 1)=((d1\rs x)\ra (dx\rs 1))\rs (dx\rs 1)$ \\
$\hspace*{2.5cm}$ $=(dx\rs ((d1\rs x)\ra 1))\rs (dx\rs 1)$ \\
$\hspace*{2.5cm}$ $\ge ((d1\rs x)\ra 1)\rs 1=(((dx\ra x)\rs x)\ra 1)\rs 1$ \\
$\hspace*{2.5cm}$ $=(dx\Cup_1 x\ra 1)\rs 1=\varphi_{dx\Cup_1 x}$. \\
On the other hand, by Lemmas \ref{psBE-40-20}$(2)$ and \ref{psBE-110}$(5)$ we have: \\
$\hspace*{2cm}$ $\varphi_{dx\Cup_1 x}=(dx\Cup_1 x\ra 1)\rs 1\ge dx\Cup_1 x\ge dx$. \\
We conclude that $dx=\varphi_{dx\Cup_1 x}$. Similarly, $dx=\varphi_{dx\Cup_2 x}$. \\
$(3)$ By $(2)$, $dx=\varphi_{dx\Cup_1 x}=(dx\Cup_1 x)\ra 1)\rs 1\ge dx\Cup_1 x\ge dx$, 
so that $dx=dx\Cup_1 x$. \\ 
Similarly, $dx=dx\Cup_2 x$. \\
$(4)$ It follows by $(2)$, since $\varphi_x\in \At(A)$, for all $x\in A;$ \\
$(5)$ Since $d1=1$, by $(3)$ and $(1)$ we get: $dx=dx\Cup_1 x=(dx\ra x)\rs x=d1\rs x =1\rs x=x$, for all $x\in A$. 
Hence $d=\Id_A$.
\end{proof}

\begin{cor} \label{sdop-50} If $d\in \mathcal{SDOP}^{(II)}(A)$, then the following hold for all $x, y\in A:$ \\
$(1)$ $\Img(d)\subseteq \At(A);$ \\
$(2)$ $dx=dx\Cup_1 y=dx\Cup_2 y;$ \\
$(3)$ $x\ra dy, x\rs dy\in \At(A)$.  
\end{cor}

\begin{prop} \label{sdop-50-10} Let $A$ be a pseudo-BCI algebra and let 
$d\in \mathcal{SDOP}^{(I)}(A)\cup \mathcal{SDOP}^{(II)}(A)$. Then $d(x \cdot y)=dx\cdot y = x\cdot dy$, for all 
$x, y\in \At(A)$.
\end{prop}
\begin{proof} Let $d\in \mathcal{SDOP}^{(I)}(A)$ and let $x, y\in \At(A)$. 
Since $1, x, y\in \At(A)$, by Corollary \ref{psBE-40-50} and Proposition \ref{sdop-40}, 
$x\ra 1, y\rs 1, dx, dy, d(x\ra 1), d(y\rs 1), d1, d1\ra x, d1\rs x\in \At(A)$. 
From $x\cdot y=(x\ra 1)\rs y=(y\rs 1)\ra x$, using Proposition \ref{psBE-40-40}$(b)$ we get: \\
$\hspace*{2cm}$   $d(x\cdot y)=d((x\ra 1)\rs y)=((x\ra 1)\rs dy)\Cup_1 (y\rs d(x\ra 1))$ \\
$\hspace*{3.3cm}$ $=(x\ra 1)\rs dy=x\cdot dy$ \\
$\hspace*{2cm}$   $d(x\cdot y)=d((y\rs 1)\ra x)=((y\rs 1)\ra dx)\Cup_2 (x\ra d(y\rs 1))$ \\
$\hspace*{3.3cm}$ $=(y\rs 1)\ra dx=dx\cdot y$. \\
Similarly, for $d\in \mathcal{SDOP}^{(II)}(A)$ we have: \\
$\hspace*{2cm}$   $d(x\cdot y)=d((x\ra 1)\rs y)=(d(x\ra 1)\rs y)\Cup_1 (dy\rs (x\ra 1))$ \\
$\hspace*{3.3cm}$ $=d(x\ra 1)\rs y=((dx\ra 1)\Cup_2 (d1\ra x))\rs y$ \\
$\hspace*{3.3cm}$ $=(dx\ra 1)\rs y=dx\cdot y$. \\
$\hspace*{2cm}$   $d(x\cdot y)=d((y\rs 1)\ra x)=(d(y\rs 1)\ra x)\Cup_2 (dx\ra (y\rs 1))$ \\
$\hspace*{3.3cm}$ $=d(y\rs 1)\ra x=((dy\rs 1)\Cup_1 (d1\rs y))\ra x$ \\
$\hspace*{3.3cm}$ $=(dy\rs 1)\ra x=x\cdot dy$. \\
Hence, in both cases $d(x \cdot y)=dx\cdot y = x\cdot dy$.  
\end{proof}

\begin{prop} \label{sdop-60} Let $A$ be a pseudo-BCI algebra and let $d_{\varphi}:A\longrightarrow A$, defined by 
$dx=\varphi_x$ for any $x\in A$. Then $d_{\varphi}\in \mathcal{SDOP}^{(I)}(A)$. 
\end{prop}
\begin{proof} 
Similar to the proof of Proposition \ref{dop-60}, since $x\ra \varphi_y\in \At(A)$, we get: \\
$\hspace*{1cm}$ $d(x\ra y)=\varphi_{x\ra y}=x\ra \varphi_y 
                   =((x\ra \varphi_y)\rs (y\ra \varphi_x))\ra (y\ra \varphi_x)$ \\
$\hspace*{2.6cm}$ $=(x\ra \varphi_y)\Cup_2 (y\ra \varphi_x)=(x\ra dy)\Cup_2 (y\ra dx)$. \\
Similarly we have: \\
$\hspace*{1cm}$ $d(x\rs y)=\varphi_{x\rs y}=x\rs \varphi_y
                          =((x\rs \varphi_y)\ra (y\rs \varphi_x))\rs (y\rs \varphi_x)$ \\
$\hspace*{2.6cm}$ $=(x\rs \varphi_y)\Cup_1 (y\rs \varphi_x)=(x\rs dy)\Cup_1 (y\rs dx)$. \\
We conclude that $d_{\varphi}\in \mathcal{SDOP}^{(I)}(A)$. 
\end{proof}

\begin{cor} \label{sdop-70} $\mathcal{IDOP}^{(I)}(A)\cap \mathcal{SDOP}^{(I)}(A)\neq \emptyset$. 
\end{cor}
\begin{proof} 
If $d_{\varphi}:A\longrightarrow A$, defined by $dx=\varphi_x$, for all $x\in A$, then by Propositions \ref{dop-60} 
and \ref{sdop-60}, $d_{\varphi}\in \mathcal{IDOP}^{(I)}(A)\cap \mathcal{SDOP}^{(I)}(A)$.  
\end{proof}

\begin{ex} \label{sdop-80} Consider the pseudo-BCI algebra $(A,\ra,\rs,1)$ and the maps 
$d_1,d_2,d_3:A\longrightarrow A$ from Example \ref{dop-30-10}. 
One can check that $d_{\varphi}=d_3$, $\mathcal{SDOP}^{(I)}(A)=\{d_2,d_3\}$ and $\mathcal{SDOP}^{(II)}(A)=\{d_3\}$.  
\end{ex}

\begin{ex} \label{sdop-90} 
Consider the structure $(A,\ra,\rs,1)$, where the operations $\ra$ and $\rs$ on $A=\{a,b,x,y,g,1\}$ 
are defined as follows:
\[
\hspace{10mm}
\begin{array}{c|cccccc}
\ra & a & b & x & y & g & 1 \\ \hline
a & 1 & b & g & y & g & 1 \\
b & a & 1 & x & g & g & 1 \\
x & g & x & 1 & a & 1 & g \\
y & y & g & b & 1 & 1 & g \\
g & y & x & b & a & 1 & g \\
1 & a & b & x & y & g & 1
\end{array}
\hspace{10mm}
\begin{array}{c|cccccc}
\rs & a & b & x & y & g & 1 \\ \hline
a & 1 & b & x & g & g & 1 \\
b & a & 1 & g & y & g & 1 \\
x & x & g & 1 & b & 1 & g \\
y & g & y & a & 1 & 1 & g \\
g & x & y & a & b & 1 & g \\
1 & a & b & x & y & g & 1
\end{array}
.
\]
Then $(A,\ra,\rs,1)$ is a pseudo-BCI algebra (\cite{Eman1}). 
Consider the maps $d_1,d_2,d_3:A\longrightarrow A$ given in the table below:
\[
\begin{array}{c|cccccc}
 x     & a & b & x & y & g & 1  \\ \hline
d_1(x) & a & b & x & y & g & 1 \\
d_2(x) & g & g & 1 & 1 & 1 & g \\
d_3(x) & 1 & 1 & g & g & g & 1   
\end{array}
.   
\]
One can see that $\mathcal{SDOP}^{(I)}(A)=\{d_2,d_3\}$ and $\mathcal{SDOP}^{(II)}(A)=\emptyset$.  
We also mention that $d_{\varphi}=d_3$ and $\mathcal{IDOP}^{(I)}(A)=\mathcal{IDOP}^{(II)}(A)=\{d_1,d_2,d_3\}$.
\end{ex}

\begin{prop} \label{sdop-110} Let $A$ be a $p$-semisimple pseudo-BCI algebra and let $d \in \mathcal{SDOP}^{(II)}(A)$. 
Then the following hold for all $x, y\in A:$ \\
$(1)$ $d(x\ra y)=dx\ra y$ and $d(x\rs y)=dx\rs y;$ \\
$(2)$ $x\ra dx=y\ra dy$ and $x\rs dx=y\rs dy;$ \\
$(3)$ $x\ra dx=dy\ra y$ and $x\rs dx=dy\rs y$. 
\end{prop}
\begin{proof}
$(1)$ It follows by Proposition \ref{psBE-40-30}$(c)$ since $A$ is $p$-semisimple. \\
$(2)$ By $(psBCI_1)$ we have: \\
$\hspace*{2cm}$ $(x\ra dx)\rs ((dx\ra y)\rs (x\ra y))=(y\ra dy)\rs ((dy\ra x)\rs (y\ra x)) (=1)$. \\
It follows by $(1)$ that: \\
$\hspace*{2cm}$ $(x\ra dx)\rs (d(x\ra y)\rs (x\ra y))=(y\ra dy)\rs (d(y\ra x)\rs (y\ra x))$. \\
Applying Proposition \ref{sdop-40}$(1)$ we have: \\
$\hspace*{2cm}$ $d(x\ra y)\rs (x\ra y)=d(y\ra x)\rs (y\ra x) (=d1)$. \\
Hence $(x\ra dx)\rs (d(x\rs y)\ra (x\ra y))=(y\ra dy)\rs (d(x\ra y)\rs (x\ra y))$. \\ 
Finally applying Proposition \ref{psBE-40-30}$(f)$ we get $x\ra dx=y\ra dy$. 
Similarly, $x\rs dx=y\rs dy$. \\
$(3)$ Since by Proposition \ref{sdop-40}$(1)$, $d1=dy\ra y$, applying $(2)$ we get 
$x\ra dx=1\ra d1=d1=dy\ra y$, that is $x\ra dx=dy\ra y$. Similarly, $x\rs dx=dy\rs  y$.
\end{proof}

\begin{prop} \label{sdop-120} If $A$ is a $p$-semisimple pseudo-BCI algebra, then 
$\mathcal{SDOP}^{(II)}(A)= \mathcal{IDOP}^{(II)}(A)$. 
\end{prop}
\begin{proof}
Let $d\in \mathcal{SDOP}^{(II)}(A)$. By Proposition \ref{sdop-110}$(1)$, since $A$ is $p$-semisimple we have 
$d(x\ra y)=dx\ra y=(dx\ra y)\Cup_2 (x\ra dy)$, that is $d\in \mathcal{IDOP}^{(II)}(A)$. 
Hence $\mathcal{SDOP}^{(II)}(A)\subseteq \mathcal{IDOP}^{(II)}(A)$. 
Conversely, if $d\in \mathcal{IDOP}^{(II)}(A)$, then 
$d(x\ra y)=(dx\ra y)\Cup_2 (x\ra dy)=dx\ra y=(dx\ra y)\Cup_2 (dy\ra x)$, that is $d\in \mathcal{SDOP}^{(II)}(A)$. 
Thus $\mathcal{IDOP}^{(II)}(A)\subseteq \mathcal{SDOP}^{(II)}(A)$. 
We conclude that $\mathcal{SDOP}^{(II)}(A)= \mathcal{IDOP}^{(II)}(A)$. 
\end{proof}

\begin{prop} \label{sdop-130} 
If $A$ is a $p$-semisimple BCI-algebra, then $\mathcal{SDOP}^{(I)}(A)=\mathcal{IDOP}^{(I)}(A)$. 
\end{prop}
\begin{proof}
Similar to \cite[Th. 3.13]{Abuj2}, based on the fact that any $p$-semisimple BCI-algebra is medial.
\end{proof}

\begin{ex} \label{sdop-140} 
Consider the structure $(A,\ra,1)$ and the maps $d_1,d_2,d_3:A\longrightarrow A$, defined in the tables below:
\[
\hspace{10mm}
\begin{array}{c|ccc}
\rightarrow & a & b & 1 \\ \hline
a & 1 & a & b \\
b & b & 1 & a \\
1 & a & b & 1
\end{array}
\hspace{10mm}
\begin{array}{c|ccc}
 x     & a & b & 1  \\ \hline
d_1(x) & a & b & 1 \\
d_2(x) & b & 1 & a \\
d_3(x) & 1 & a & b 
\end{array}
.
\]
Then $(A,\ra,1)$ is a $p$-semisimple BCI-algebra (\cite{Asl1}). 
One can check that $\mathcal{SDOP}^{(I)}(A)=\mathcal{IDOP}^{(I)}(A)=\{d_1,d_2,d_3\}$ and 
$\mathcal{SDOP}^{(II)}(A)=\mathcal{IDOP}^{(II)}(A)=\{d_1\}$.  
\end{ex}

\begin{theo} \label{sdop-150} Let $A$ be a pseudo-BCI algebra and let 
$d\in \mathcal{SDOP}^{(I)}(A)\cup \mathcal{SDOP}^{(II)}(A)$. Then $d$ is regular if and only 
if every deductive system of $A$ is $d$-invariant.   
\end{theo}
\begin{proof}
Let $d\in \mathcal{SDOP}^{(I)}(A)\cup \mathcal{SDOP}^{(II)}(A)$ and assume that every 
$D\in \mathcal {DS}(A)$ is $d$-invariant. Since $\{1\}\in \mathcal {DS}(A)$, it follows that 
$d(\{1\})\subseteq \{1\}$, hence $d(1)=1$, that is $d$ is regular. 
Conversely, let $D\in \mathcal {DS}(A)$ and let $y\in d(D)$, that is there exists $x\in D$ such that $y=dx$. \\
If $d\in \mathcal{SDOP}^{(I)}(A)$, then by Proposition \ref{sdop-30}$(3)$ we have $x\ra y=x\ra dx=1\in D$, 
hence $y\in D$. 
For $d\in \mathcal{SDOP}^{(II)}(A)$, by Proposition \ref{sdop-40}$(5)$ we get $x\ra y=x\ra dx=x\ra x=1$, 
that is $y\in D$. 
It follows that $d(D)\subseteq D$, hence $D$ is $d$-invariant. 
\end{proof}

\bigskip

\section{Conclusions and future work}

In this paper we introduce and study two concepts of implicative derivation operators on pseudo-BCI algebras: 
type I implicative derivation defined by conditions $(idop_1)$, $(idop_2)$ and type II implicative derivation 
defined by conditions $(idop_3)$, $(idop_4)$. 
These conditions were required by the proof of Proposition \ref{dop-40}, which is crucial for the results of 
this paper. For the particular case of the pseudo-BCK algebras the above mentioned result is also valid for 
another two types of implicative derivations: \\ 
$-$ \emph{type III implicative derivation} defined by the following conditions, for all $x, y\in A:$ \\
$(idop_5)$ $d(x\ra y)=(x\ra dy)\Cup_1 (dx\ra y)$ \\
$(idop_6)$ $d(x\rs y)=(x\rs dy)\Cup_2 (dx\rs y)$, \\
$-$ \emph{type IV implicative derivation} defined by the following conditions, for all $x, y\in A:$ \\ 
$(idop_7)$ $d(x\ra y)=(dx\ra y)\Cup_1 (x\ra dy)$ \\
$(idop_8)$ $d(x\rs y)=(dx\rs y)\Cup_2 (x\rs dy)$. \\
In what follows we give an example of these derivations, but the investigation of type III and type IV implicative derivations on pseudo-BCK algebras is the topic of another work (\cite{Ciu100}). \\
Consider the structure $(A,\ra,\rs,1)$, where the operations $\ra$ and $\rs$ on $A=\{0,a,b,c,1\}$ 
are defined as follows:
\[
\hspace{10mm}
\begin{array}{c|ccccc}
\rightarrow & 0 & a & b & c & 1 \\ \hline
0 & 1 & 1 & 1 & 1 & 1 \\
a & 0 & 1 & b & 1 & 1 \\
b & a & a & 1 & 1 & 1 \\
c & 0 & a & b & 1 & 1 \\
1 & 0 & a & b & c & 1
\end{array}
\hspace{10mm}
\begin{array}{c|ccccc}
\rightsquigarrow & 0 & a & b & c & 1 \\ \hline
0 & 1 & 1 & 1 & 1 & 1 \\
a & b & 1 & b & 1 & 1 \\
b & 0 & a & 1 & 1 & 1 \\
c & 0 & a & b & 1 & 1 \\
1 & 0 & a & b & c & 1
\end{array}
.
\]
Then $(A,\ra,\rs,1)$ is a pseudo-BCK algebra (\cite{Ciu4}). 
Consider the maps $d_i:A\longrightarrow A$, $i=1,\cdots,6$, given in the table below:
\[
\begin{array}{c|cccccc}
 x     & 0 & a & b & c & 1  \\ \hline
d_1(x) & 0 & a & b & c & 1 \\
d_2(x) & 0 & a & b & 1 & 1 \\
d_3(x) & 0 & 1 & 1 & c & 1 \\
d_4(x) & 0 & 1 & 1 & 1 & 1 \\
d_5(x) & 1 & 1 & 1 & c & 1 \\
d_6(x) & 1 & 1 & 1 & 1 & 1   
\end{array}
.   
\]
One can check that $\mathcal{IDOP}^{(I)}(A)=\mathcal{IDOP}^{(II)}(A)=\mathcal{IDOP}^{(IV)}(A)=\{d_1,d_2,d_5,d_6\}$ and 
$\mathcal{IDOP}^{(III)}(A)=\{d_1,d_2,d_3,d_4,d_5,d_6\}$. \\ 
We also introduce two types of symmetric derivations on pseudo-BCI algebras and study the relationship between 
implicative and symmetric derivations. \\
The results presented in this paper could be extended to other ``pseudo" algebras such as pseudo-BCH algebras, pseudo-BE algebras, pseudo-CI algebras, non-commutative residuated lattices. 
As a generalization of the concept of derivations, the $f$-derivation has been defined for lattices (\cite{Cev1}), BCI-algebras (\cite{Jav1}, \cite{Nis1}, \cite{Zhan1}), BE-algebras (\cite{Kim2}). 
If $(A,*,0)$ is a BCI-algebra and $f$ is an endomorphism of $A$, then a map $d:A\longrightarrow A$ is called a \emph{left-right $f$-derivation} if $d(x*y)=(dx*f(y))\wedge (f(x)*dy)$ and a \emph{right-left $f$-derivation}
if $d(x*y)=(f(x)*dy)\wedge (dx*f(y))$ for all $x, y\in A$, where $x\wedge y=y*(y*x)$.
As another direction of research one could define and study the concept of $f$-derivations on pseudo-BCI algebras 
and other ``pseudo" algebras. \\
The concept of generalized derivations on BCI-algebras defined and investigated in \cite{Ozt1}, \cite{Muh1}, \cite{Muh2} could be also introduced and studied for the case of pseudo-BCI algebras. \\
The BCI-algebras with product (BCI(P)-algebras) $(A,\odot, \ra, 1)$ were originally introduced by 
$\rm Is\acute{e}ki$ (\cite{Ise1}) as BCI-algebras with condition (S). 
The concept of \emph{multiplicative derivations} on BCI(P)-algebras could be defined and studied as another 
topic of future research.



\bigskip

\noindent {\footnotesize
\begin{minipage}[b]{10cm}
Lavinia Corina Ciungu\\
Department of Mathematics \\
University of Iowa \\
14 MacLean Hall, Iowa City, Iowa 52242-1419, USA \\
Email: lavinia-ciungu@uiowa.edu
\end{minipage}}


\bigskip

\begin{thebibliography}{99}

\bibitem{Abuj1} H.S.A. Abujabal, N.O. Alshehri, \emph{Some results on derivations of BCI-algebras}, Coden Jnsmac  \textbf{42}(2006), 13--19.

\bibitem{Abuj2} H.S.A. Abujabal, N.O. Alshehri, \emph{On left derivations of BCI-algebras}, Soochow J. Math. \textbf{33}(2007), 435--444.

\bibitem{Als1} N.O. Alshehri, \emph{Derivations on MV-algebras}, Int. J. Math. Math. Sci. \textbf{2010}(2010), 312027.

\bibitem{Als2} N.O. Alshehri, S.M. Bawazeer, \emph{On derivations of BCC-algebras}, Int. J. Algebra  \textbf{6}(2012), 1491--1498.

\bibitem{Asl1} M. Aslam, A.B. Thaheem, \emph{A note on $p$-semisiple BCI-algebras}, Math. Japon. {\bf 36}(1991), 39--45.

\bibitem{Cev1} Y. \c Ceven, M.A. $\rm \ddot{O}zt\ddot{u}rk$, \emph{On $f$-derivation of lattices}, Bull. Korean Math. Soc. \textbf{45}(2008), 701–-707.

\bibitem{Chaj1} I. Chajda, \emph{A structure of pseudo-BCI algebras}, Int. J. Theor. Phys. \textbf{53}(2014), 
3391–-3396.

\bibitem{Chaj2} I. Chajda, H. $\rm L\ddot{a}nger$, \emph{An ordered structure of pseudo-BCI algebras}, 
Math. Bohemica \textbf{141}(2016), 91–-98.


\bibitem{Ciu4} L.C. Ciungu, \emph{Non-commutative Multiple-Valued Logic Algebras}, Springer, Cham, Heidelberg,  New York, Dordrecht, London, 2014. 

\bibitem{Ciu100} L.C. Ciungu, \emph{Derivation operators on generalized algebras of BCK logic}, submitted.  


\bibitem{Dud1} W.A. Dudek, Y.B. Jun, \emph{Pseudo-BCI algebras}, East Asian Math. J. {\bf 24}(2008), 187--190.

\bibitem{Dym1} G. Dymek, \emph{On two classes of pseudo-BCI algebras}, Discuss. Math., Gen. Algebra Appl. \textbf{31}(2011), 217–-229.

\bibitem{Dym2} G. Dymek, \emph{p-semisimple pseudo-BCI algebras}, J. Mult.-Valued Logic Soft Comput. \textbf{19}(2012), 461–-474.

\bibitem{Dym8} G. Dymek, \emph{Atoms and ideals of pseudo-BCI algebras}, Comment. Math. \textbf{52}(2012), 73–-90.

\bibitem{Dym3} G. Dymek, A. Kozanecka-Dymek, \emph{Pseudo-BCI logic}, Bull. Sect. Logic. \textbf{42}(2013), 33–-41.

\bibitem{Dym4} G. Dymek, \emph{On compatible deductive systems of pseudo-BCI algebras}, J. Mult.-Valued Logic Soft Comput. \textbf{22}(2014), 167–-187.

\bibitem{Dym5} G. Dymek, \emph{On a period of elements of pseudo-BCI algebras}, Discuss. Math., Gen. Algebra Appl. \textbf{35}(2015), 21–-31.

\bibitem{Dym6} G. Dymek, \emph{On a period part of pseudo-BCI algebras}, Discuss. Math., Gen. Algebra Appl. \textbf{35}(2015), 139–-157.

\bibitem{Dym7} G. Dymek, \emph{On pseudo-BCI algebras}, Ann. Univ. Mariae Curie-Sk\l odowska Sect. A 
\textbf{LXIX}(2015), 59–-71.

\bibitem{Eman1} P. $\rm Emanovsk\acute{y}$, J. $\rm K \ddot{u}hr$, \emph{Some properties of pseudo-BCK- and 
pseudo-BCI algebras}, Fuzzy Sets Syst. {\bf 339}(2018), 1--16.


\bibitem{Geo15} G. Georgescu, A. Iorgulescu, \emph{Pseudo-BCK algebras: An extension of BCK-algebras}, Proceedings of DMTCS'01: Combinatorics, Computability and Logic, Springer, London, 2001, pp. 97--114.


\bibitem{Ghor1} Sh. Ghorbani, L. Torkzadeh, S. Motamed, \emph{$(\odot,\oplus)$-derivations and $(\ominus,\odot)$-derivations on MV-algebras}, Iran. J. Math. Sci. Inform. {\bf 1}(2013), 75--90.

\bibitem{Fer1} L. Ferrari, \emph{On derivations of lattices}, Pure Math. Appl. {\bf 12}(2001), 365--382.

\bibitem{Jav1} M.A. Javed, M. Aslam, \emph{A note on $f$-derivations of BCI-algebras}, Commun. Korean Math. Soc. 
{\bf 24}(2009), 321--331.

\bibitem{He1} P. He, X. Xin, J. Zhan, \emph{On derivations and their fixed point sets in residuated lattices}, 
Fuzzy Sets Syst. {\bf 303}(2016), 97--113.

\bibitem{Ior1} A. Iorgulescu, \emph{Classes of pseudo-BCK algebras - Part I}, J. Mult.-Valued Logic Soft Comput. 
{\bf 12}(2006), 71--130.

\bibitem{Ior14} A. Iorgulescu, \emph{Algebras of logic as BCK-algebras}, ASE Ed., Bucharest, 2008.

\bibitem{Ior15} A. Iorgulescu, \emph{Implicative-groups vs. groups and generalizations}, Matrix Rom Ed., 
Bucharest, 2018.

\bibitem{Ise1} K. $\rm Is\acute{e}ki$, \emph{On BCI-algebras with condition (S)}, Math. Sem. Notes {\bf 8}(1980), 171--172.

\bibitem{Jun1} Y.B. Jun, X.L. Xin, \emph{On derivations on BCI-algebras}, Inf. Sci. {\bf 159}(2004), 167--176.

\bibitem{Jun2} Y.B. Jun, H.S. Kim, J. Neggers, \emph{On pseudo-BCI ideals of pseudo-BCI algebras}, 
Mat. Vesnik {\bf 58}(2006), 39--46.

\bibitem{Kim1} K.H. Kim, S.M. Lee,  \emph{On derivations of BE-algebras}, Honam Math. J. \textbf{36}(2014), 167--178.

\bibitem{Kim2} K.H. Kim, B. Davvaz,  \emph{On $f$-derivations of BE-algebras}, J. Chungcheong Math. Soc.  \textbf{28}(2015), 127--138.

\bibitem{Krna1} J. $\rm Kr\check{n}\acute{a}vek$, J. $\rm K \ddot{u}hr$,  \emph{A note on derivations on basic algebras}, Soft Comput. \textbf{19}(2015), 1765--1771.

\bibitem{Kuhr6} J. $\rm K \ddot{u}hr$, \emph{Pseudo-BCK algebras and related structures}, Habilitation thesis, 
$\rm Palack\acute{y}$ University in Olomouc, 2007. 

\bibitem{Lee1} K.J. Lee, C.H. Park, \emph{Some ideals of pseudo-BCI algebras}, J. Appl. Math. Inform.  \textbf{27}(2009), 217--231. 

\bibitem{Muh2} G. Muhiuddin, A.M. Al-Roqi, \emph{On generalized left derivations in BCI-algebras}, Appl. Math. Inf. Sci. \textbf{8}(2014), 1153--1158.

\bibitem{Muh1} G. Muhiuddin, A.M. Al-Roqi, \emph{Generalizations of derivations in BCI-algebras}, Appl. Math. Inf. 
Sci. \textbf{9}(2015), 89--94.

\bibitem{Nis1} F. Nisar,  \emph{On $f$-derivations of BCI-algebras}, J. Prime Res. Math. \textbf{5}(2009), 176--191.

\bibitem{Ozt1} M.A. $\rm \ddot{O}zt\ddot{u}rk$, Y. \c Ceven, Y.B. Jun,  \emph{Generalized derivations of BCI-algebras}, Honam Math. J. \textbf{31}(2009), 601--609.

\bibitem{Pos1} E. Posner, \emph{Derivations in prime rings}, Proc. Am. Math. Soc. \textbf{8}(1957), 1093--1100.

\bibitem{Prab1} C. Prabpayak, U. Leerawat, \emph{On Derivations of BCC-algebras}, Kasetsart J. (Nat. Sci.)  \textbf{43}(2009), 398--401. 

\bibitem{Rac7} J. Rach{\accent23u}nek, D. $\rm \check{S}alounov\acute{a}$, \textit{Derivations on algebras of a non-commutative generalization of the \L ukasiewicz logic}, Fuzzy Sets Syst. \textbf{333}(2018), 11--16. 

\bibitem{Sza1} G. $\rm Sz\acute{a}sz$, \emph{Derivations of lattices}, Acta Sci. Math. (Szeged) \textbf{37}(1975), 149--154.

\bibitem{Xin1} X.L. Xin, T.Y. Li, J.H. Lu, \emph{On derivations of lattices}, Inf. Sci. {\bf 178}(2008), 307--316.

\bibitem{Xin2} X.L. Xin, \emph{The fixed set of a derivation in lattices}, Fixed Point Theory Appl. 
{\bf 218}(2012), 1--12.

\bibitem{Yaz1} H. Yazarli,  \emph{A note on derivations on MV-algebras}, Miskolc Math. Notes \textbf{14}(2013), 345--354.

\bibitem{Zhan1} J. Zhan, Y.L. Liu,  \emph{On $f$-derivations of BCI-algebras}, Int. J. Math. Math. Sci.  \textbf{11}(2005), 1675--1684.




\end{thebibliography}
\end{document}